\documentclass[12pt]{article}
\usepackage{amsmath,amssymb}
\usepackage{amsthm}

\def\BC{\mathcal{B}}

\def\E{\mathbf{E}}

\def\P{\mathbf{P}}
\def\R{\mathbf{R}}

\def\1{\mathbf{1}}

\def\Int{\rm{Int}}

\def\al{\alpha}
\def\be{\beta}
\def\pa{\partial}
\def\ep{\epsilon}
\def\de{\delta}
\def\ga{\gamma}
\def\ka{\varkappa}

\newtheorem{prop}{Proposition}[section]
\newtheorem{theorem}{Theorem}[section]

\newtheorem{corollary}{Corollary}
\newtheorem{remark}{Remark}

\newcommand{\la}{\lambda}
\newcommand{\si}{\sigma}

\newcommand{\om}{\omega}

\newcommand{\Om}{\Omega}

\begin{document}
\title{Game theoretic analysis of incomplete markets: emergence of probabilities,
nonlinear and fractional Black-Scholes equations
\thanks{}
\thanks{Supported by the AFOSR grant FA9550-09-1-0664 'Nonlinear Markov control processes and games'}}
\author{Vassili N. Kolokoltsov\thanks{Department of Statistics, University of Warwick,
 Coventry CV4 7AL UK,
  Email: v.kolokoltsov@warwick.ac.uk}}
\maketitle

\begin{abstract}
Expanding the ideas of the author's paper \cite{Ko98} we develop a pure game-theoretic
approach to option pricing, by-passing stochastic modeling. Risk neutral probabilities emerge
automatically from the robust control evaluation. This approach seems to be especially appealing
for incomplete markets encompassing extensive, so to say untamed, randomness, when the coexistence
of infinite number of risk neutral measures precludes one from unified pricing of derivative securities.
Our method is robust enough to be able to accommodate various markets rules and settings including path dependent
payoffs, American options and transaction costs. On the other hand, it leads to rather simple numerical
algorithms. Continuous time limit is described by nonlinear and/or fractional Black-Scholes type equations.
\end{abstract}



\section{Introduction}
\label{secintr}

Expanding the ideas of the author's papers \cite{Ko98}, \cite{Ko99} we develop a pure game-theoretic
approach to option pricing in a multi-dimensional market (rainbow options), where risk neutral probabilities emerge
automatically from the robust control evaluation.
The process of investment is considered as a zero-sum game of an investor with the Nature.

 For basic examples of complete markets, like
binomial model or geometric Brownian motion, our approach yields the same results as the classical
 (by now) risk neutral evaluation developed by Cox-Ross-Rubinstein or Black-Scholes. However, for
 incomplete markets, like for rainbow options in multi-dimensional binomial or interval models, the coexistence
of infinite number of risk neutral measures precludes one from unified pricing of derivative securities
by usual methods. Several competing methods were proposed for pricing options
under these circumstances (see e.g. a review in Bingham and Kiesel \cite{BiRu}), most of them using
certain subjective criteria, say a utility function for payoff or a certain risk measure. The difference in
pricing arising from these methods is justified by referring vaguely to the intrinsic risk of incomplete
markets. In our game-theoretic approach, no subjectivity enters the game. We define and calculate a hedge
price, which is the minimal capital needed to meet the obligation for all performances of
the markets, within the rules specified by the model (dominated hedging).

Though our price satisfies the so called
'no strictly acceptable opportunities' (NSAO) condition suggested in
Carr, Geman and Madan \cite{Carr01},
one still may argue of course that this is not a completely fair price,
as the absence of an exogenously specified initial probability distribution does not allow
us to speak about a.s. performance and implies per force a possibility
of an additional surplus. To address this issue,
we observe that together with the hedging price for buying a security,
that may be called an upper price, one can equally
reasonable define a lower price, which can be looked as a hedge for selling the security. The difference of these two
values can be considered as a precise measure of the intrinsic risk that
 is incident to incomplete markets.
 An alternative way to deal with possible unpredictable surplus, as suggested e.g. in Lyons
 \cite{TerLy} for models with unknown volatility, consists in specifying a possible cash-back, which
 should be due to the holder of an option when the moves of the prices (unpredictable at the beginning) turn out to
 be favorable.

 Our method is robust enough to be able to accommodate various markets rules and settings including path dependent
payoffs, American options, real options and transaction costs. Continuous time limit is described by nonlinear and/or fractional Black-Scholes type equations.

As a possible weakness of our approach we should mention that,
in order to be effective, eligible movements of a
market should be reasonably bounded. Possible big jump
should be taken into account separately, say by means of the theory
of extreme values.

Brief content of the paper is as follows.
In Section \ref{Options} we set a stage by defining the game of an investor with the Nature leading
to the basic game theoretic expression for the hedging price in the simplest case of
a standard European (rainbow) option without transaction costs taken into account.

In the next three sections, which are completely independent of any financial applications,
 we carry out a preparatory work on evaluating certain
rather abstract minmax expressions, showing in particular, how naturally risk neutral probabilities emerge
(and more precisely the extreme points of these probabilities), as if by miracle,
from minimizing Legendre transforms of concave functions defined on polyhedrons.

In Section \ref{secbacktooptions} we apply these results for the evaluation of hedge prices in the simplest
setting. Section \ref{secsubmodular} shows the essential simplifications that become available
for sub-modular payoffs. In particular, a unique risk -neutral selector
can be specified sometimes, say in case of two colored options (for a still incomplete market).
This is of crucial importance, as the major examples of real-life rainbow payoffs
turn out to be sub-modular. Section \ref{sectransco} shows how transaction costs can be nicely fit
into our model. Next two sections are devoted to the modifications needed for more complex models
including path dependent payoffs, American options and transaction costs.
Only in case of precisely $J+1$ possible jumps of a $J$-dimensional vector of stock process
the corresponding market becomes complete.

Section \ref{secupperlowerprice} introduces the dual formulations and explicit expressions for
upper and lower hedging prices. Next two sections are devoted to continuous time limits.
These limits are obtained again without any probability, but only assuming that the magnitude of jumps
per time $\tau$ is of order $\tau^{\al}$, $\al \in [1/2,1]$. Finally,
in Section \ref{secfracdyn}, the model with waiting times having power decay is discussed showing that
its limit is described by a fractional (degenerate and/or nonlinear) version of Black-Scholes equation.

Some bibliographical comments seem to be in order.
Game-theoretic (or robust control) approach for options was used in
McEneaney \cite{Mc}, though in this paper the main point was in proving that the option prices of standard models
can be characterized as viscosity solutions of the corresponding Hamilton-Jacobi equation. As a
by-produce it was confirmed (similarly to analogous results in
Avellaneda, Levy and Par\'as\cite{AveLe} and Lyons \cite{TerLy})
 that  one can hedge prices in stochastic volatility models by the
Black-Scholes strategies specified by the maximal volatility. A related paper is Olsder \cite{Ol}, where only a basic
one-dimensional model was analyzed, though with some transaction costs included.

The reasonability of the extension of the binomial model allowing for price jumps inside the interval
(interval model) was realized by several authors, see Kolokoltsov \cite{Ko98}, Bernard \cite{Bern05},
Aubin, Pujal and Saint-Pierre \cite{AuPu} and Roorda, Engwerda and Schumacher
\cite{RoEnSch}. In the latter paper the term interval model was coined.
The series of papers of P. Bernard et al \cite{Bern05}, \cite{Bern05a}, \cite{Bern07} deals with one-dimensional models with very general strategies and transaction costs
including both continuous and jump-type trading. Arising Hamilton-Jacobi-Bellman
equation have peculiar degeneracies that require subtle techniques to handle.

Hedging by domination (super-replication), rather than replication, is well establish in the literature, especially
in connection with models incorporating transaction costs, see e.g. \cite{BaSo}.
Problems with transaction costs in standard models are well known, as indicates the title 'There
is no non trivial hedging portfolio for option pricing with transaction costs' of the
paper Soner, Shreve and Cvitani\'c
\cite{SSC}. This problem, similar to the story with incomplete markets, leads to the development
of optimizations based on a subjectively chosen utility function, see e.g. Davis and Norman \cite{DaNo} or
Barles and Soner \cite{BaSo}.

Upper and lower values for prices were discussed in many places, see e.g.
El Karoui and Quenez \cite{ElKa} or Roorda, Engwerda and Schumacher.
\cite{RoEnSch}. An abstract definition of lower and upper prices can be given
in the general game-theoretic approach to probability and finances advocated in monograph
Shafer and Vovk \cite{ShV}.

The well known fact that the existing (whatever complicated) stochastic models are far from being
precise reflections of the real dynamics of market prices leads naturally to the attempts to relax the
assumed stochastic restrictions of models. For instance, Avellaneda, Levy and Par\'as\cite{AveLe}
 and Lyons \cite{TerLy} work with unknown volatilities leading
to nonlinear Black-Scholes type equations (though still non-degenerate, unlike those
obtained below). On the other hand, Hobson \cite{Hob98} (see also
\cite{Hob10}, \cite{Hob10a} and references therein) suggests model independent estimates based on the observed
prices of traded securities, the main technique being the Skorohod embedding problem (SEP).
  These approaches still build the theory on some basic
underlying stochastic model (e. g. geometric Brownian motion), unlike our method that starts upfront
with the robust control. Similarly, hedging with respect to several (or all) equivalent martingale
measures, based on the optional decomposition (see
 F\"ollmer and Kramkov \cite{FolKra}
and Kramkov \cite{Kra96}), are based on some initial probability law (with respect to which equivalence
is considered). The risk-neutral or martingale measures that arise from our approach are not linked
to any initial law. They are not equivalent, but represent extreme points of risk-neutral measures on all
possible realizations of a stock price process.

'Fractional everything' becomes a popular topic in modern literature, see e.g. the recent monograph
 Tarasov \cite{Ta11}. For the study of financial markets,
 this is of course a natural step to move from the discussion of power laws in economics (see e.g.
various perspectives in Uchaikin and Zolotarev \cite{UZ}, Newman \cite{New06}, Maslov \cite{Mas05}
and references therein) to the applicability of fractional dynamics in financial markets, see e.g.
Meerschaert and Scala \cite{MSca}, Meerschaert, Nane and Xiao \cite{MSNaXi},
Jumarie \cite{Jum},  Wang \cite{Wa10} and references therein.
Our game-theoretic analysis leads to degenerate and/or nonlinear versions of
fractional Black-Scholes type equations.

{\bf Notations.} By $|z|$ we denote the magnitude (Euclidean norm) of a vector $z$ and by $\|f\|$ the sup-norm
of a function.
We shall denote by $\circ$ the point-wise multiplication of vectors (sometimes
called {\it Hadamard} multiplication):
\[
(y \circ z)^i =y^i z^i.
\]

{\bf Acknowledgements.} The author is grateful to Sigurd Assing, Alain Bensoussan,
David Hobson, Alex Mijatovic and Oleg Malafeyev for useful comments, and to Pierre Bernard for encouraging him
to delve deeper in the topic of the paper.

\section{Colored options as a game against Nature}
\label{Options}

Recall that a European {\it option}\index{option} is a contract between
two parties where one party has right to complete a transaction in
the future (with previously agreed amount, date and price) if he/
she chooses, but is not obliged to do so.
More precisely, consider a financial market dealing with several securities:
the risk-free bonds (or bank account) and $J$ common stocks,
$J=1,2...$. In case $J>1$, the corresponding options are called {\it
colored or rainbow options}\index{option!rainbow or colored}
($J$-colors option for a given $J$). Suppose the prices of the units of these
securities, $B_m$ and $S_m^{i}$, $i\in \{1,2,...,J\}$, change in
discrete moments of time $m=1,2,...$
according to the recurrent equations $B_{m+1}=\rho B_m$, where the
$\rho \geq 1$ is an interest rate which remains unchanged over time,
and $S_{m+1}^{i}=\xi_{m+1}^{i}S_m^i$, where $\xi_m^{i},i\in
\{1,2,...,J\}$, are unknown sequences taking values in some fixed
intervals $M_{i}=[d_{i},u_{i}]\subset \R$. This model
generalizes the colored version of the classical CRR model in a
natural way. In the latter a sequence $\xi_m^{i}$ is confined to
take values only among two boundary points $d_{i},u_{i}$, and it is
supposed to be random with some given distribution. In our model any
value in the interval $[d_{i},u_{i}]$ is allowed and no probabilistic assumptions are made.

The type of an option is specified by a given premium function $f$
of $J$ variables. The following are the standard examples:

option delivering the best of $J$ risky assets and cash
\begin{equation}
f(S^{1},S^{2},...,S^{J})=\max (S^{1},S^{2},...,S^{J},K),
 \label{best risky assets}
\end{equation}

calls on the maximum of $J$ risky assets
\begin{equation}
f(S^{1},S^{2},...,S^{J})=\max (0,\max (S^{1},S^{2},...,S^{J})-K),
\label{calls on max}
\end{equation}

multiple-strike\index{option!multiple-strike} options
\begin{equation}
f(S^{1},S^{2},...,S^{J})=\max
(0,S^{1}-K_{1},S^{2}-K_{2},....,S^{J}-K_{J}),
\label{multiple strike options}
\end{equation}

portfolio\index{option!portfolio} options
\begin{equation}
f(S^{1},S^{2},...,S^{J})=\max
(0,n_{1}S^{1}+n_{2}S^{2}+...+n_{J}S^{J}-K),
\end{equation}

and spread\index{option!spread} options
\begin{equation}
f(S^{1},S^{2})=\max (0,(S^{2}-S^{1})-K).
\end{equation}

Here, the $S^{1},S^{2},...,S^{J}$ represent the (in principle
unknown at the start) expiration date values of the underlying
assets, and $K,K_{1},...,K_{J}$ represent the (agreed from the
beginning) strike prices. The presence of $\max$ in all these
formulae reflects the basic assumption that the buyer is not obliged
to exercise his/her right and would do it only in case of a positive
gain.

The investor is supposed to control the growth of his/her capital in the
following way. Let $X_m$ denote the capital of the investor at the
time $m=1,2,...$. At each time $m-1$ the investor determines his
portfolio by choosing the numbers $\ga _m^i$ of common stocks
of each kind to be held so that the structure of the capital is
represented by the formula

\[
X_{m-1} =\sum_{j=1}^{J}\ga_m^j S_{m-1}^j
 +(X_{m-1}-\sum_{j=1}^{J}\ga_m^j S_{m-1}^j),
\]
where the expression in bracket corresponds to the part of his
capital laid on the bank account. The control parameters $\ga_m^j$
 can take all real values, i.e. short selling and borrowing are
allowed. The value $\xi_m$ becomes known in the moment $m$ and thus
the capital at the moment $m$ becomes

\begin{equation}
\label{eqnewcap1}
X_m=\sum_{j=1}^{J}\ga _m^j\xi_m^j S_{m-1}^j +\rho
(X_{m-1}-\sum_{j=1}^J \ga_m^j S_{m-1}^j),
\end{equation}
if transaction costs are not taken into account.

If $n$ is the prescribed {\it maturity date}\index{option!maturity
date}, then this procedures repeats $n$ times starting from some
initial capital $X=X_0$ (selling price of an option) and at the end
the investor is obliged to pay the premium $f$ to the buyer. Thus
the (final) income of the investor equals
\begin{equation}
\label{eqnewcap2}
G(X_n,S_n^1,S_n^2,...,S_n^J)=X_n-f(S_n^1,S_n^2,...,S_n^J).
\end{equation}

The evolution of the capital can thus be described by the $n$-step
game of the investor with the Nature, the behavior of the latter
being characterized by unknown parameters $\xi_m^j$. The strategy of
the investor is by definition any sequences of vectors
$(\ga_1,\cdots,\ga_n)$
 such that each
$\ga_m$ could be chosen using the whole previous
information: the sequences $X_{0},...,X_{m-1}$ and $S_0^i,...,S_{m-1}^j$
(for every stock $j=1,2,...,J$).
 The control parameters $\ga_m^j$ can take all real values,
i.e. short selling and borrowing are allowed. A position of the game
at any time $m$ is characterized by $J+1$ non-negative numbers
$X_m,S_m^1, \cdots ,S_m^J$ with the final income specified by the
function
\begin{equation}
G(X,S^{1},...,S^{J})=X-f(S^{1},...,S^{J})
\label{eqG function}
\end{equation}

The main definition of the theory is as follows. A strategy $\ga_1,\cdots,\ga_n$, of the investor is called a
{\it hedge}\index{hedge}, if for any sequence
  $(\xi_1, \cdots ,\xi_n)$
the investor is able to meet his/her obligations, i.e.
\[
G(X_{n},S_{n}^{1},...,S_{n}^{J})\geq 0.
\]
The minimal value of the capital $X_{0}$ for which the hedge exists
is called the {\it hedging price} $H$ of an option.

Looking for the guaranteed payoffs means looking
for the worst case scenario (so called {\it robust control
approach}\index{robust control}), i.e. for the minimax strategies. Thus if the
final income is specified by a function $G$, the guaranteed income
of the investor in a one step game with the initial conditions
$X,S^{1},...,S^{J}$ is given by the {\it Bellman
operator}\index{Bellman operator}
\begin{equation}
\label{eqBellmanforopnonred}
\mathbf{B}G(X, S^1,\cdots, S^J)
 =\frac{1}{\rho}\max_{\ga}\min_{\{\xi^j \in [d_j,u_j]\}}
 G(\rho X+ \sum_{i=1}^{J}\ga^i\xi^iS^{i}
  -\rho \sum_{i=1}^{J}\ga^iS^i,\xi^1S^1, \cdots,\xi^JS^J),
\end{equation}
and (as it follows from the standard backward induction argument, see e.g. \cite{Bel60} or
\cite{KoMabook}) the guaranteed income of
the investor in the $n$ step game with the initial conditions
$X_{0},S_{0}^{1},...,S_{0}^{J}$ is given by the formula

\[
\mathbf{B}^{n}G(X_{0},S_{0}^{1},...,S_{0}^{J}).
\]

In our model $G$ is given by \eqref{eqG function}. Clearly for $G$ of the form
\[
G (X,S^1,\cdots, S^J)=X-f(S^1,\cdots,S^J),
\]
\[
\mathbf{B}G(X,S^{1},...,S^{J})
 =X - \frac{1}{\rho}\min_{\ga}\max_{\xi}
 [f(\xi ^1 S^1,\xi^2 S^2, \cdots,\xi^J S^J)
-\sum_{j=1}^J\ga^j S^j(\xi^j-\rho)],
\]
and hence
\[
\mathbf{B}^nG(X,S^1,\cdots,S^J)
= X -\frac{1}{\rho^n}(\BC^{n}f)(S^1, \cdots ,S^J),
\]
where the {\it reduced Bellman operator} is defined as:
\begin{equation}
\label{eqBellmanforop}
(\BC f)(z^1,...,z^J)=\min_{\ga}\max_{\{\xi^j \in [d_j,u_j]\}}
[f(\xi ^1 z^1,\xi^2 z^2, \cdots,\xi^J z^J)
-\sum_{j=1}^J\ga^j z^j(\xi^j-\rho)],
\end{equation}
or, in a more concise notations,
\begin{equation}
\label{eqBellmanforop2}
(\BC f)(z)=\min_{\ga}\max_{\{\xi^j \in [d_j,u_j]\}}
[f(\xi \circ z)-(\ga, \xi \circ z-\rho z)].
\end{equation}

This leads to the following result from \cite{Ko98}.

\begin{theorem}
\label{thoptionprice}
 The minimal value of $X_{0}$ for which the
income of the investor is not negative (and which by definition is
the hedge price $H^n$ in the $n$-step game) is given by
\begin{equation}
H^{n}=\frac{1}{\rho^{n}}(\BC^{n}f)(S_0^1, \cdots , S_0^J).
 \label{hedgeprice}
\end{equation}
\end{theorem}

We shall develop a method for evaluating the operator \eqref{eqBellmanforop},
as well as its modifications for American options or when transaction costs are taken into account.

\section{Underlying game-theoretic setting}
\label{secunderlyinggt}

In this section we develop a general technique for the evaluation of minmax expressions of type
\eqref{eqBellmanforop2} showing how naturally the extreme risk neutral probabilities arise
in such evaluation. We also supply geometric estimations for these probabilities  and the corresponding
minimizing value of $\ga$, which are crucial for a nonlinear extension given in Section \ref{secnonlinext}.
In order to explain the ideas clearly, we first develop the theory in dimension $d=2$, and then extend it
to arbitrary dimensions (which requires certain lengthy manipulation with multidimensional determinants).

We shall denote by $\Int$ the interior of a closed set.
Let a closed convex polygon in $\R^2$ contains the origin as an interior point,
 and let $\xi_1, \cdots , \xi_k$ be its vertices, ordered anticlockwise. We shall denote such a polygon by $\Pi=\Pi[\xi_1, \cdots , \xi_k]$. The assumed condition
\begin{equation}
\label{eqzeroint}
0\in {\Int} \Pi[\xi_1, \cdots , \xi_k]
\end{equation}
implies that all $\xi_i$ do not vanish.

We are interested in the following game-theoretic problem: find
\begin{equation}
\label{eq1mainminmaxop}
\Pi[\xi_1, \cdots , \xi_k](f)=\min_{\ga \in \R^2} \max_{\xi \in \Pi} [f(\xi)-(\xi,\ga)]
\end{equation}
for a convex (possibly non strictly) function $f$. By convexity, this rewrites as
\begin{equation}
\label{eq2mainminmaxop}
\Pi[\xi_1, \cdots , \xi_k](f)=\min_{\ga \in \R^2} \max_{\xi_1,\cdots, \xi_k} [f(\xi_i)-(\xi_i,\ga)].
\end{equation}

Having this in mind, we shall analyze a slightly more general problem: for an arbitrary finite collection
of non-vanishing vectors $\xi_1, \cdots , \xi_k$ from $\R^2$, ordered anticlockwise, and
arbitrary numbers $f(\xi_1),\cdots, f(\xi_k)$, to calculate \eqref{eq2mainminmaxop}
(whenever the minimum exists). The corresponding polygon $\Pi[\xi_1, \cdots , \xi_k]$ (obtained by linking
together all neighboring vectors $\xi_i, \xi_{i+1}$,
$i=1,\cdots,k$, and $\xi_k, \xi_1$ with straight segments) may not be convex anymore.

We shall start with the case of $\Pi$ being a triangle: $\Pi=\Pi[\xi_1,\xi_2,\xi_3]$.
Then condition \eqref{eqzeroint}
implies that $\xi_i\neq -\al \xi_j$ for $\al >0$ and any $i,j$.
Suppose the $\min$ in
\begin{equation}
\label{eq3mainminmaxop}
\Pi[\xi_1, \xi_2, \xi_3](f)=\min_{\ga \in \R^2} \max_{\xi_1,\xi_2, \xi_3} [f(\xi_i)-(\xi_i,\ga)]
\end{equation}
is attained on a vector $\ga_0$ and the corresponding $\max$ on a certain $\xi_i$. Suppose this $\max$ is unique,
so that
 \begin{equation}
\label{eq1mainminmaxopsol}
 f(\xi_i)-(\xi_i,\ga)>  f(\xi_j)-(\xi_j,\ga)
\end{equation}
for all $j\neq i$. As $\xi_i\neq 0$, by changing $\ga_0$ on a small amount we can reduce the l.h.s. of
\eqref{eq1mainminmaxopsol} by preserving the inequality \eqref{eq1mainminmaxopsol}. This possibility contradicts
the assumption that $\ga_0$ is a minimal point. Hence, if  $\ga_0$ is a minimal point, the corresponding maximum must be attained on at least two vectors. Suppose it is attained on precisely two vectors, that is
  \begin{equation}
\label{eq2mainminmaxopsol}
 f(\xi_i)-(\xi_i,\ga)=f(\xi_j)-(\xi_j,\ga)>f(\xi_m)-(\xi_m,\ga)
\end{equation}
for some different $i,j,m$. Since the angle between $\xi_i,\xi_j$ is strictly less than $\pi$, adding a vector
\[
\ep (\xi_i/|\xi_j|+\xi_j/|\xi_i|)
\]
 to $\ga_0$
will reduce simultaneously first two expressions from the l.h.s. of \eqref{eq2mainminmaxopsol}, but preserve (for small enough $\ep$) the inequality on the r.h.s. of \eqref{eq2mainminmaxopsol}. This again
contradicts
the assumption that $\ga_0$ is a minimal point. Hence, if  $\ga_0$ is a minimal point, it must
satisfy the equation
\begin{equation}
\label{eq3mainminmaxopsol}
f(\xi_1)-(\xi_1,\ga)=f(\xi_2)-(\xi_2,\ga)=f(\xi_3)-(\xi_3,\ga),
\end{equation}
which is equivalent to the system
   \begin{equation}
\label{eq4mainminmaxopsol}
\left\{
\begin{aligned}
 (\xi_2-\xi_1, \ga_0)=f(\xi_2)-f(\xi_1),
  \\
  (\xi_3-\xi_1, \ga_0)=f(\xi_3)-f(\xi_1).
\end{aligned}
\right.
\end{equation}
Again by assumption
\eqref{eqzeroint}, the vectors $\xi_2-\xi_1, \xi_3-\xi_1$ are independent. Hence system
\eqref{eq4mainminmaxopsol} has a unique solution $\ga_0$.

For a pair of vectors $u,v\in \R^2$, let $D(u,v)$ denote the oriented area of the parallelogram built on $u,v$ and $R(u)$ the result of the rotation of $u$ on $90^{\circ}$ anticlockwise. That is, for $u=(u^1,u^2)$, $v=(v^1,v^2)$,
\[
D(u,v)=u^1v^2-u^2v^1, \quad R(u)=(u^2,-u^1).
\]
Notice that the determinant of system \eqref{eq4mainminmaxopsol} is
\[
D(\xi_2-\xi_1,\xi_3-\xi_1)=D(\xi_2,\xi_3)+D(\xi_3,\xi_1)+D(\xi_1,\xi_2),
\]
and by the standard formulas of linear algebra, the unique solution $\ga_0$ is
 \begin{equation}
 \label{eq5mainminmaxopsol}
 \ga_0=\frac{f(\xi_1) R(\xi_2-\xi_3)+f(\xi_2) R(\xi_3-\xi_1)+f(\xi_3) R(\xi_1-\xi_2)}
 {D(\xi_2,\xi_3)+D(\xi_3,\xi_1)+D(\xi_1,\xi_2)},
\end{equation}
and the corresponding optimal value
 \begin{equation}
 \label{eq6mainminmaxopsol}
 \Pi[\xi_1,\xi_2,\xi_3](f)=\frac{f(\xi_1) D(\xi_2,\xi_3)+f(\xi_2) D(\xi_3,\xi_1)+f(\xi_3) D(\xi_1,\xi_2)}
 {D(\xi_2,\xi_3)+D(\xi_3,\xi_1)+D(\xi_1,\xi_2)}.
\end{equation}
Hence we arrive at the following.

\begin{prop}
\label{propthreepointriskneutral}
Let a triangle $\Pi[\xi_1,\xi_2,\xi_3]$ satisfy \eqref{eqzeroint}, and let
$f(\xi_1), f(\xi_2), f(\xi_3)$ be arbitrary numbers. Then expression \eqref{eq3mainminmaxop} is given by
\eqref{eq6mainminmaxopsol} and the minimum is attained on the single $\ga_0$ given by
\eqref{eq5mainminmaxopsol}.
\end{prop}

\begin{proof}
Our discussion above shows that if $\ga_0$ is a minimum point, then it is unique and given by \eqref{eq5mainminmaxopsol}.
It remains to show that this $\ga_0$ is in fact the minimal point. But this is straightforward, as any change in $\ga_0$
would necessarily increase one of the expressions $f(\xi_i)-(\xi_i,\ga)$ (which again follows from
\eqref{eqzeroint}). Alternatively, the same conclusion can be obtained indirectly from the observation
that the minimum exists and is attained on some finite $\ga$, because
\[
  \max_{\xi_1,\xi_2,\xi_3} [f(\xi_i)-(\xi_i,\ga)] \to \infty,
  \]
as $\ga \to \infty$.
\end{proof}

\begin{corollary}
Expression \eqref{eq6mainminmaxopsol} can be written equivalently as
\[
  \Pi[\xi_1,\xi_2,\xi_3](f)=\E f(\xi),
  \]
  where the expectation is defined with respect to the probability law $\{p_1,p_2,p_3\}$ on $\xi_1,\xi_2,\xi_3$:
  \[
  p_i=\frac{D(\xi_j,\xi_m)}{D(\xi_2,\xi_3)+D(\xi_3,\xi_1)+D(\xi_1,\xi_2)}
  \]
  ($(i,j,k)$ is either (1,2,3) or (2,3,1) or (3,1,2)). Moreover, this distribution is the unique
  probability on $\xi_1,\xi_2,\xi_3$ such that
   \begin{equation}
 \label{eq7mainminmaxopsol}
 \E(\xi)=\sum_{i=1}^3 p_i \xi_i=0.
 \end{equation}
  \end{corollary}

\begin{proof}
Required uniqueness follows from the uniqueness of the expansion of $\xi_3$ with respect to the basis $\xi_1,\xi_2$.
\end{proof}

  We shall call a probability law on $\xi_1,\xi_2,\xi_3$ {\it risk-neutral},
  if it satisfies \eqref{eq7mainminmaxopsol}. The reason for this terminology will be seen later.
  From the point of view of convex analysis this is just a probability on $\xi_1,\xi_2,\xi_3$
  with barycenter in the origin.

  We can now calculate \eqref{eq1mainminmaxop} for arbitrary $k$.
 \begin{theorem}
\label{thpolygonriskneutral}
 Let a polygon $\Pi=\Pi[\xi_1, \cdots, \xi_k]$ satisfy
 the following conditions:

(i) No two vectors $\xi_i, \xi_j$ are linearly dependent;

(ii) The collection $\{\xi_1, \cdots, \xi_k\}$ does not belong to any half-space, i.e. there is no $\om \in \R^2$
such that $(\om,\xi_i)>0$ for all $i$.

Then
\begin{equation}
\label{eq8mainminmaxopsol}
\Pi[\xi_1, \cdots , \xi_k](f)=\max_{i,j,m} \E_{ijm}f(\xi)
= \max_{i,j,m} (p_i^{ijm} f(\xi_i)+p_j^{ijm} f(\xi_j)+p_m^{ijm} f(\xi_m)),
\end{equation}
where $\max$ is taken over all triples $1\le i <j<m \le k$ such that
\begin{equation}
\label{eq9mainminmaxopsol}
0\in {\Int} \Pi [\xi_i,\xi_j,\xi_k],
\end{equation}
and $\{p_i^{ijm},p_j^{ijm},p_m^{ijm}\}$ denotes the unique risk neutral probability
on $\{\xi_i,\xi_j, \xi_m\}$ (given by Proposition \ref{propthreepointriskneutral})
with $\E_{ijm}$ the corresponding expectation.
\end{theorem}

\begin{remark}
Condition (i) is equivalent to the geometrical requirement that
the origin does not lie on any diagonal of $\Pi$ (or its extension), and condition (ii) is equivalent to
\eqref{eqzeroint}.
\end{remark}

\begin{proof}
 For any triple $\{i,j,m\}$ satisfying \eqref{eq9mainminmaxopsol},
\[
\Pi[\xi_1, \cdots , \xi_k](f)
\ge \min_{\ga \in \R^2} \max_{\xi_i,\xi_j,\xi_m} [f(\xi)-(\xi,\ga)]
=\E_{ijm}f(\xi),
\]
where Proposition \ref{propthreepointriskneutral} was used for the last equation. Hence
  \begin{equation}
\label{eq10mainminmaxopsol}
\Pi[\xi_1, \cdots , \xi_k](f)\ge \max_{i,j,m} \E_{ijm}f(\xi).
\end{equation}

A key geometrical observation is the following. Conditions (i) and (ii) imply that there exists
a subset of the collection $\{\xi_1, \cdots, \xi_k\}$ consisting only of three vectors $\{\xi_i,\xi_j,\xi_m\}$,
but still satisfying these conditions (and hence the
assumptions of Proposition \ref{propthreepointriskneutral}).
This follows from the Carath\'eodory theorem (but can be also seen directly, as one can take an arbitrary
$\xi_i$, and then choose, as $\xi_m,\xi_j$, the vectors with the maximum angle (less than $\pi$) with $\xi_i$ when rotating clockwise and anticlockwise respectively).  This observation implies that the maximum on the r.h.s of
\eqref{eq10mainminmaxopsol} is defined (the set of triples is not empty) and consequently the l.h.s. is bounded from below.
Moreover, as for any triple  $\{i,j,m\}$ satisfying \eqref{eq9mainminmaxopsol},
\[
\max_{\xi_i,\xi_j,\xi_m} [f(\xi)-(\xi,\ga)] \to \infty,
\]
as $\ga \to \infty$, and hence also
\[
\max_{i=1,\dots,k} [f(\xi_i)-(\xi_i,\ga)] \to \infty,
  \]
 the minimum in \eqref{eq2mainminmaxop} is attained on some finite $\ga$.
Assuming that $\ga_0$ is such a minimum point, we can now argue as above to conclude that
\begin{equation}
\label{eq11mainminmaxopsol}
 f(\xi_i)-(\xi_i,\ga_0)=f(\xi_j)-(\xi_j,\ga_0)=f(\xi_m)-(\xi_m,\ga_0)
\end{equation}
for some triple $\xi_i,\xi_j,\xi_m$. Moreover, if these triple does not satisfy (i) and (ii), then
(by the same argument) the l.h.s. of \eqref{eq11mainminmaxopsol} can not strictly exceed
$f(\xi_l)-(\xi_l,\ga_0)$ for all other $\xi_l$. Hence we are led to a conclusion that if $\ga_0$ is a minimum
point, then there exists a subset $I\subset\{1,\cdots , k\}$ such that the expressions $f(\xi_l)-(\xi_l,\ga_0)$
coincide for all $l\in I$ and the family $\{\xi_l\}$, $l\in I$, satisfy conditions (i), (ii). But by the above
geometrical observation, such a family has to contain a subfamily with three vectors only satisfying (i) and (ii).
Consequently, \eqref{eq8mainminmaxopsol} holds.
\end{proof}

\begin{remark}
It is easy to see that the number of allowed triples $\{i,j,m\}$ on the r.h.s. of
\eqref{eq8mainminmaxopsol} is two for $k=4$, can be $3,4$ or $5$ (depending on the position of the origin
inside $\Pi$) for $k=5$, and can be $4,6$ or $8$ for $k=6$. This number seems to increase exponentially,
as $k\to \infty$.
\end{remark}

\begin{remark} Theorem \ref{thpolygonriskneutral} can be easily extended to the situation when
conditions (i) and/or (ii) are not satisfied. Namely, if (ii) does not hold, then the l.h.s. of
\eqref{eq2mainminmaxop} is not defined (equals to $-\infty$). If (i) does not hold, then
the $\max$ on the r.h.s of \eqref{eq8mainminmaxopsol} should be over all eligible triples plus all risk
neutral expectations over all pairs such that $\xi_i=-\al \xi_j$, $\al>0$.
\end{remark}

Let us extend the results to higher dimensions $d$. Let us start with the
simplest case of $d+1$ vectors $\xi_1,\cdots,\xi_{d+1}$ in $\R^d$. Suppose their convex hull
$\Pi[\xi_1, \cdots , \xi_{d+1}]$ is such that
\begin{equation}
\label{eq2mainminmaxoparbd}
0\in {\Int} \Pi[\xi_1, \cdots , \xi_{d+1}].
\end{equation}

We are interested in evaluating the expression
\begin{equation}
\label{eq1mainminmaxoparbd}
\Pi[\xi_1, \cdots , \xi_{d+1}](f)=\min_{\ga \in \R^d} \max_i [f(\xi_i)-(\xi_i,\ga)].
\end{equation}
A remarkable fact that we are going to reveal is linear in $f$ and the minimizing $\ga$ is
unique and also depends linearly on $f$.

Assume that $\R^d$ is equipped with the standard basis $e_1, \cdots, e_d$ fixing the orientation.
Without loss of generality we shall assume now that the vectors $\xi_1, \cdots , \xi_{d+1}$
are ordered in such a way that the vectors $\{\xi_2, \xi_3, \cdots, \xi_{d+1}\}$ form an oriented basis of $\R^d$.
The fact that the vector $\xi_1$ lies outside any half space containing
this basis, allows one to identify the orientation of other subsets of $\xi_1, \cdots , \xi_{d+1}$
of size $d$. Namely, let $\{\hat \xi_i \}$ denote the ordered subset of $\xi_1, \cdots , \xi_{d+1}$
obtained by taking $\xi_i$ out of it. The basis $\{\hat \xi_i \}$ is oriented if and only if $i$ is odd.
For instance, if $d=3$, the oriented bases form the triples $\{\xi_2, \xi_3, \xi_4\}$, $\{\xi_1, \xi_2, \xi_4\}$,
 $\{\xi_1, \xi_4, \xi_3\}$ and $\{\xi_1, \xi_3, \xi_2\}$.

The same argument as for $d=2$ leads us to the conclusion that a minimal point $\ga_0$ must
satisfy the equation
   \begin{equation}
\label{eq3mainminmaxoparbd}
 f(\xi_1)-(\xi_1,\ga)=\cdots =f(\xi_{d+1})-(\xi_{d+1},\ga),
\end{equation}
which is equivalent to the system
   \begin{equation}
\label{eq4mainminmaxoparbd}
 (\xi_i-\xi_1, \ga_0)=f(\xi_i)-f(\xi_1), \quad i=2, \cdots, d+1.
\end{equation}
From \eqref{eq2mainminmaxoparbd} it follows that
 this system
 has a unique solution, say $\ga_0$.

To write it down explicitly, we shall use the natural extensions of the notations used above for $d=2$.
For a collection of $d$ vectors $u_1, \cdots, u_d \in \R^d$, let $D(u_1, \cdots, u_d)$ denote the oriented volume of the parallelepiped built on $u_1, \cdots, u_d$ and $R(u_1, \cdots, u_{d-1})$ the rotor of the family
 $(u_1, \cdots, u_{d-1})$. That is, denoting by upper scripts the coordinates of vectors,
\[
D(u_1, \cdots, u_d)=\det \left(
 \begin{aligned}
& u^1_1 \quad \cdots \quad u^d_1  \\
& u^1_2 \quad \cdots \quad u^d_2  \\
& \quad \quad \cdots \quad \quad \\
& u^1_d \quad \cdots \quad u^d_d
\end{aligned}
\right),
 \quad R(u_1, \cdots, u_{d-1})=\det \left(
 \begin{aligned}
& e_1 \quad \quad \cdots \quad \quad e_d  \\
& u_1^1 \quad \quad \cdots \quad \quad u_1^d \\
& \quad \quad \quad \cdots \quad \quad \\
& u_{d-1}^1 \quad \cdots \quad u_{d-1}^d
\end{aligned}
 \right)
 \]
 \[
 = e_1 \det \left(
 \begin{aligned}
& u_1^2 \quad \quad \cdots \quad \quad u_1^d  \\
& \quad \quad \quad \cdots \quad \quad \\
& u_{d-1}^2 \quad \cdots \quad u_{d-1}^d
\end{aligned}
\right)
-e_2 \det \left(
 \begin{aligned}
& u_1^1 \quad \quad u_1^3 \quad \cdots \quad \quad u_1^d  \\
& \quad \quad \quad \cdots \quad \quad \\
& u_{d-1}^1 \quad u_{d-1}^3 \quad \cdots \quad u_{d-1}^d
\end{aligned}
\right)
+ \cdots .
\]

Finally, let us define a poly-linear operator $\tilde R$ from an ordered collection
$\{u_1,\cdots, u_d\}$ of $d$ vectors in $\R^d$ to $\R^d$:
\[
\tilde R(u_1,\cdots, u_d)
=R(u_2-u_1,u_3-u_1, \cdots, u_d-u_1)
\]
\[
=R(u_2,\cdots, u_d)-R(u_1,u_3, \cdots, u_d)+\cdots
+ (-1)^{d-1} R(u_1,\cdots, u_{d-1}).
\]

Returning to system \eqref{eq4mainminmaxoparbd} observe that its determinant,
 which we denote by $D$, equals
\[
D=D(\xi_2-\xi_1,\cdots, \xi_{d+1}-\xi_1)
 =\det \left(
 \begin{aligned}
& \xi_2^1-\xi_1^1 \quad \quad \xi_2^2-\xi_1^2 \quad \cdots \quad \quad \xi_2^d-\xi_1^d  \\
& \quad \quad \quad \quad \cdots \quad \\
& \xi_{d+1}^1-\xi_1^1 \quad \xi_{d+1}^2-\xi_1^2 \quad \cdots \quad \xi_{d+1}^d-\xi_1^d
\end{aligned}
\right)
\]
Using the linear dependence of a determinant on columns, this rewrites as
\[
D(\xi_2,\cdots, \xi_{d+1})
 -\xi_1^1 \det \left(
 \begin{aligned}
& 1 \quad \, \, \xi^2_2 \quad \cdots \quad \,\, \xi^d_2  \\
& \quad \quad \cdots \quad   \\
& 1 \quad \xi^2_{d+1} \quad \cdots \quad \xi_{d+1}^d
\end{aligned}
\right)
-\xi_1^2 \det \left(
 \begin{aligned}
& \xi_2^1 \quad \quad 1 \quad \xi^3_2 \quad \cdots \quad \quad \xi^d_2  \\
& \quad \quad \cdots \quad   \\
& \xi_{d+1}^1 \quad 1 \quad \xi^3_{d+1} \quad \cdots \quad \xi_{d+1}^d
\end{aligned}
\right)
 -\cdots,
 \]
 implying that
\begin{equation}
\label{eq5mainminmaxoparbd}
D=D(\xi_2-\xi_1,\cdots, \xi_{d+1}-\xi_1)
=\sum_{i=1}^{d+1} (-1)^{i-1} D(\{\hat \xi_i \}).
\end{equation}

Notice that according to the orientation specified above,  $D(\{\hat \xi_i \})$
are positive (resp. negative) for odd $i$ (resp. even $i$), implying that
all terms in \eqref{eq5mainminmaxoparbd} are positive, so that the collection
of numbers
\begin{equation}
\label{eq6mainminmaxoparbd}
p_i=\frac1D (-1)^{i-1} D(\{\hat \xi_i \})
=\frac{(-1)^{i-1} D(\{\hat \xi_i \})}{D(\xi_2-\xi_1,\cdots, \xi_d-\xi_1)},
\quad i=1,\cdots, d+1,
\end{equation}
define a probability law on the set $\xi_1,\cdots,\xi_{d+1}$ with a full support.

By linear algebra, the unique solution $\ga_0$
to system \eqref{eq4mainminmaxoparbd} is given by the formulas
\begin{equation}
\label{eq71mainminmaxoparbd}
\ga_0^1= \frac1D \det \left(
 \begin{aligned}
& f(\xi_2)-f(\xi_1) \quad \quad \xi^2_2-\xi_1^2 \quad \quad \cdots \quad \quad \xi^d_2-\xi^d_1  \\
& \quad \quad \cdots \quad   \\
& f(\xi_{d+1})-f(\xi_1) \quad  \xi^2_{d+1} -\xi_1^2 \quad \cdots \quad \xi_{d+1}^d-\xi^d_1
\end{aligned}
\right),
\end{equation}
\begin{equation}
\label{eq72mainminmaxoparbd}
 \ga_0^2= \frac1D \det \left(
 \begin{aligned}
& \xi^1_2-\xi_1^1 \quad \quad f(\xi_2)-f(\xi_1)  \quad \quad \cdots \quad \quad \xi^d_2-\xi^d_1  \\
& \quad \quad \cdots \quad   \\
& \xi^1_{d+1} -\xi_1^1 \quad \quad f(\xi_{d+1})-f(\xi_1) \quad \cdots \quad \xi_{d+1}^d-\xi^d_1
\end{aligned}
\right),
\end{equation}
and similar for other $\ga_0^i$.
One sees by inspection that for any $i$
\begin{equation}
 \label{eq7mainminmaxoparbd}
f(\xi_i)-(\ga_0,\xi_i)
 =\frac1D \sum_{i=1}^{d+1} [f(\xi_i)
(-1)^{i+1} D(\{\hat \xi_i \})],
\end{equation}
and
 \[
 \ga_0=\frac1D (f(\xi_2)-f(\xi_1)) R(\xi_3-\xi_1,\cdots, \xi_{d+1}-\xi_1)
 -\frac1D (f(\xi_3)-f(\xi_1)) R(\xi_2-\xi_1,\xi_4-\xi_1, \cdots, \xi_{d+1}-\xi_1)
 \]
 \[
 +\cdots + \frac1D (-1)^{d+1} (f(\xi_{d+1})-f(\xi_1) R(\xi_2-\xi_1,\cdots, \xi_d-\xi_1),
\]
which rewrites as
\begin{equation}
 \label{mainminmaxopgaarbd}
 \ga_0=-\frac1D \left[
 f(\xi_1) \tilde R(\{ \hat \xi_1 \})
 -f(\xi_2) \tilde R(\{ \hat \xi_2 \})
 +\cdots + (-1)^d f(\xi_{d+1})\tilde R(\{ \hat \xi_{d+1} \}),
\right]
\end{equation}

For example, in case $d=3$, we have
\[
 \Pi[\xi_1,\cdots,\xi_4](f)=\frac{f(\xi_1) D_{234}+f(\xi_2) D_{143}+f(\xi_3) D_{124}+f(\xi_4) D_{132}}
 {D_{234}+D_{143}+D_{124}+ D_{132}},
\]
\[
 \ga_0=-\frac{f(\xi_1) R_{234}+f(\xi_2) R_{143}+f(\xi_3) R_{124}+f(\xi_4) R_{132}}
 {D_{234}+D_{143}+D_{124}+ D_{132}},
\]
where $D_{ijm}= D(\xi_i, \xi_j, \xi_m)$ and
\[
R_{ijm}=R(\xi_i,\xi_j)+R(\xi_j,\xi_m)+R(\xi_m,\xi_i).
\]

As in case $d=2$, we arrive at the following.

\begin{prop}
\label{propriskneutrald}
Let a family $\{\xi_1, \cdots ,\xi_{d+1}\}$ in $\R^d$  satisfy condition \eqref{eq2mainminmaxoparbd},
 and let
$f(\xi_1)$, $\cdots ,f(\xi_{d+1})$ be arbitrary numbers. Then
\begin{equation}
 \label{eq73mainminmaxoparbd}
 \Pi[\xi_1,\cdots,\xi_{d+1}](f)
 =\frac1D \sum_{i=1}^{d+1} [f(\xi_i)
(-1)^{i+1} D(\{\hat \xi_i \})],
\end{equation}
and the minimum in \eqref{eq1mainminmaxoparbd} is attained on the single $\ga_0$ given by
\eqref{mainminmaxopgaarbd}.
\end{prop}

\begin{corollary}
\label{cor1topropriskneutrald}
Under the assumptions of Proposition \ref{propriskneutrald},
\begin{equation}
 \label{eq8mainminmaxoparbd}
  \Pi[\xi_1,\cdots,\xi_{d+1}](f)=\E f(\xi),
\end{equation}
\begin{equation}
 \label{eq81mainminmaxoparbd}
  \ga_0=\E \left[ f(\xi) \frac{\tilde R (\{\hat \xi \})}{D(\{\hat \xi \})}\right],
\end{equation}
where the expectation is with respect to the probability law
\eqref{eq6mainminmaxoparbd}.
This law is the unique risk neutral
probability law on $\{\xi_1,\cdots,\xi_{d+1}\}$, i.e. the one satisfying
\begin{equation}
 \label{eq9mainminmaxoparbd}
 \E(\xi)=\sum_{i=1}^{d+1} p_i \xi_i=0.
 \end{equation}
  \end{corollary}

\begin{proof}
The only thing left to prove is that the law \eqref{eq6mainminmaxoparbd} satisfies
\eqref{eq9mainminmaxoparbd}. But as the r.h.s. of
\eqref{eq9mainminmaxoparbd} can be written as the vector-valued determinant
\[
\E f(\xi)
=\det \left(
 \begin{aligned}
& \xi_1 \quad \xi_2 \quad \cdots \quad \xi_{d+1}  \\
& \xi_1^1 \quad \xi_2^1 \quad \cdots \quad \xi_{d+1}^1  \\
& \quad \quad \cdots \quad \quad \\
& \xi_1^d \quad \xi_2^d \quad \cdots \quad \xi_{d+1}^d
\end{aligned}
\right),
\]
that is a vector with co-ordinates
\[
\det \left(
 \begin{aligned}
& \xi_1^j \quad \xi_2^j \quad \cdots \quad \xi_{d+1}^j  \\
& \xi_1^1 \quad \xi_2^1 \quad \cdots \quad \xi_{d+1}^1  \\
& \quad \quad \cdots \quad \quad \\
& \xi_1^d \quad \xi_2^d \quad \cdots \quad \xi_{d+1}^d
\end{aligned}
\right), \quad j=1, \cdots, d.
\]
it clearly vanishes.
\end{proof}

To better visualize the above formulas, it is handy to delve a bit into
their geometric meaning.
Each term $(-1)^{i-1} D(\{\hat \xi_i \})$ in \eqref{eq5mainminmaxoparbd} equals $d!$ times the volume
of the pyramid (polyhedron) with vertices $\{0 \cup \{\hat \xi_i \}\}$. The determinant $D$,
being the volume of the parallelepiped built on $\xi_2-\xi_1,\cdots, \xi_{d+1}-\xi_1$, equals
$d!$ times the volume of the pyramid $\Pi [\xi_1, \cdots, \xi_{d+1}]$
in the affine space $\R^d$ with vertices being the end points
of the vectors $\xi_i$, $i=1,\cdots, d+1$. Consequently,
formula \eqref{eq5mainminmaxoparbd} expresses the decomposition of the volume of the pyramid
$\Pi [\xi_1, \cdots, \xi_{d+1}]$ into $d+1$ parts, the volumes of the pyramids $\Pi[\{0 \cup \{\hat \xi_i \}\}]$ obtained by sectioning from the origin, and the weights of the distribution \eqref{eq6mainminmaxoparbd}
are the ratios of these parts to the whole volume. Furthermore,
the magnitude of the rotor $R(u_1,\cdots, u_{d-1})$ is known to equal the volume of the parallelepiped
built on $u_1,\cdots, u_{d-1}$. Hence $\|\tilde R(\{\xi_i\})\|$ equals $(d-1)!$ times the volume
(in the affine space $\R^d$) of the $(d-1)$-dimensional face
of the pyramid $\Pi [\xi_1, \cdots, \xi_{d+1}]$ with vertices $\{\hat \xi_i \}$.
Hence the magnitude of the ratios $\tilde R (\{\hat \xi_i \})/ D(\{\hat \xi_i \})$,
 playing the roles of weights in \eqref{eq81mainminmaxoparbd}, are
 the ratios of the $(d-1)!$ times $(d-1)$-dimensional volumes of the bases of the pyramids
  $\Pi[\{0 \cup \{\hat \xi_i \}\}]$ to the $d!$ times their full $d$-dimensional volumes.
  Consequently,
  \begin{equation}
 \label{eqweightratioforga}
  \frac{\|\tilde R (\{\hat \xi_i \})\|}{D(\{\hat \xi_i \})}
  = \frac{1}{h_i},
\end{equation}
where $h_i$ is the length of the perpendicular from the origin to the affine hyperspace generated by the
end points of the vectors  $\{\hat \xi_i \}$. These geometric considerations lead
directly to the following estimates for
expressions \eqref{eq8mainminmaxoparbd} and \eqref{eq81mainminmaxoparbd}.

\begin{corollary}
\label{cor2topropriskneutrald}
\begin{equation}
 \label{eq1estimforPiandga}
  |\Pi[\xi_1,\cdots,\xi_{d+1}](f)| \le \| f \|,
\end{equation}
\begin{equation}
 \label{eq2estimforPiandga}
  |\ga_0| \le \| f \| \max_{i=1,\cdots, d+1} h_i^{-1},
\end{equation}
with $h_i$ from \eqref{eqweightratioforga}.
\end{corollary}

These estimates are of importance for numerical calculations of $\ga_0$
(yielding some kind of stability estimates with respect to the natural parameters).
 On the other hand, we shall need them
for nonlinear extensions of Proposition \ref{propriskneutrald} discussed later.

Let us say that a finite family of non-vanishing vectors $\xi_1,\cdots, \xi_k$
in $\R^d$ are in general position, if the following conditions hold (extending naturally
 the corresponding conditions used in case $d=2$):

(i) No $d$ vectors out of this family are linearly dependent,

(ii) The collection $\{\xi_1, \cdots, \xi_k\}$ does not belong to any half-space, i.e. there is no $\om \in \R^d$
such that $(\om,\xi_i)>0$ for all $i$.

\begin{remark}
In Roorda, Schumacher and Engwerda \cite{RoSchEn}, condition (ii) is called {\it positive completeness} of the family $\{\xi_1, \cdots, \xi_k\}$.
\end{remark}

It is worth noting that in case $k=d+1$, assuming (i) and (ii) is equivalent to \eqref{eq2mainminmaxoparbd}.

We are interested in evaluating the expression
\begin{equation}
\label{eq10mainminmaxoparbd}
\Pi[\xi_1, \cdots , \xi_k](f)=\min_{\ga \in \R^d} \max_i [f(\xi_i)-(\xi_i,\ga)].
\end{equation}

 \begin{theorem}
\label{thriskneutrald}
Let a family of non-vanishing vectors $\xi_1,\cdots, \xi_k$ in $\R^d$ satisfy (i) and (ii).
Then
\begin{equation}
\label{eq11mainminmaxoparbd}
\Pi[\xi_1, \cdots , \xi_k](f)=\max_{\{I\}} \E_{I}f(\xi)
\end{equation}
where $\max$ is taken over all families $\{\xi_i\}_{i\in I}$, $I\subset \{1,\cdots, k\}$
of size $|I|=d+1$ that satisfy (ii) (i.e. such that the origin is contained in the interior of
$\Pi [\xi_i, \, i\in I]$),
and $\E_{I}$ denotes the expectation with respect to the unique risk neutral probability
on $\{\xi_i\}_{i\in I}$ (given by Proposition \ref{propriskneutrald}).
\end{theorem}

\begin{proof}
This is the same as the proof of Theorem \ref{thpolygonriskneutral}. The key geometrical observation, that any
subset of the family
$\{\xi_1, \cdots, \xi_k\}$ satisfying (i) and (ii) contains necessarily a subset with precisely $d+1$ elements
still satisfying (ii), is a direct consequence of the Carath\'eodory theorem.
\end{proof}

\begin{remark}
\label{remlinearpartofBelop}
As is easy to see, the $\max$ in \eqref{eq11mainminmaxoparbd}
is attained on a family $\{\xi_i\}_{i\in I}$ if and only if
\begin{equation}
\label{eq12mainminmaxoparbd}
f(\xi_i)-(\ga_I,\xi_i)\ge f(\xi_r)-(\ga_I, \xi_r)
\end{equation}
for any $i\in I$ and any $r$, where $\ga_I$ is the corresponding optimal value.
 Consequently, on the convex set of functions $f$ satisfying
inequalities \eqref{eq12mainminmaxoparbd} for all $r$, the mapping
$\Pi[\xi_1, \cdots , \xi_k](f)$ is linear:
\[
\Pi[\xi_1, \cdots , \xi_k](f)=\E_If(\xi).
\]
\end{remark}

\section{Extreme points of risk-neutral laws}

We shall expand a bit on the geometrical interpretation of the above results.

Let us call a probability law $\{p_1, \cdots, p_k\}$ on a finite set $\{\xi_1, \cdots, \xi_k\}$
of vectors in $\R^d$ risk-neutral
(with respect to the origin) if the origin is its barycenter, that is

 \begin{equation}
 \label{eqriskneutdefd}
 \E(\xi)=\sum_{i=1}^k p_i \xi_i=0.
 \end{equation}

The geometrical interpretation we have in mind follows from the following simple observation.

\begin{prop}
\label{propextremriskneutr}
For a family $\{\xi_1, \cdots, \xi_k\}$ satisfying (i) and (ii),
the extreme points of the convex set of risk-neutral probabilities
 are risk-neutral probabilities with supports on subsets of size precisely $d+1$,
 satisfying themselves conditions (i) and (ii).
\end{prop}

\begin{proof}
It is clear that risk-neutral probabilities with supports on subsets of size precisely $d+1$,
 satisfying themselves conditions (i) and (ii) are extreme points.
 In fact, if this were not the case for such a probability law, then it could be presented as
 a convex combination of other risk-neutral laws. But these risk-neutral laws would
 necessarily have the same support as the initial law, which would contradict the uniqueness
 of the risk-neutral law supported on $d+1$ points in general position.

  Assume $p=(p^1,\dots, p^m)$
 is a risk-neutral probability law on $m>d+1$ points $\xi_1,\cdots, \xi_m$.
 Linear dependence of the vectors $\xi_2-\xi_1,\cdots, \xi_m-\xi_1$ implies the existence
 of a non-vanishing vector $b=(b^1,\dots, b^m)$ in $\R^m$ such that
 \[
 \sum_{i=1}^m b^i=0, \quad \sum_{i=1}^m b^i \xi_i=0.
 \]
 Hence for small enough $\ep$, the vectors
 $p-\ep b$ and $p+\ep b$ are risk neutral probability laws on $\xi_1,\cdots, \xi_m$.
 But
 \[
 p=\frac12(p-\ep b)+\frac12(p+\ep b),
 \]
 showing that $p$ is not an extreme point.
 \end{proof}

 Proposition \ref{propextremriskneutr}
 allows one to reformulate Theorem
 \ref{thriskneutrald}
 in the following way.

  \begin{theorem}
\label{thpolyhedralriskneutrald}
Let a family of non-vanishing vectors $\xi_1,\cdots, \xi_k$ in $\R^d$ satisfy (i) and (ii).
Then the r.h.s. of formula \eqref{eq11mainminmaxoparbd}, i.e.
\begin{equation}
\label{eq1thpolyhedralriskneutrald}
\Pi[\xi_1, \cdots , \xi_k](f)
=\min_{\ga \in \R^d} \max_i [f(\xi_i)-(\xi_i,\ga)]=\max \E_I f(\xi)
\end{equation}
can be interpreted as the maximum of the averages of $f$ with respect to all extreme points
of the risk-neutral probabilities on $\xi_1,\cdots, \xi_k$.
All these extreme probabilities are expressed in a closed form,
given by \eqref{eq6mainminmaxoparbd}.
\end{theorem}

It is natural to ask, what happens if conditions (i) or (ii) do not hold.
If (ii) does not hold, then $\Pi[\xi_1, \cdots , \xi_k](f)$ is not defined (equals to $-\infty$).
If only (i) does not hold, one just has to take into account possible additional extreme risk neutral probabilities coming from projections to subspaces. This leads to the following result obtained as a straightforward
extension of Theorem \ref{thpolyhedralriskneutrald}.

 \begin{theorem}
\label{thpolyhedralriskneutrald2}
Let a family of non-vanishing vectors $\xi_1,\cdots, \xi_k$ in $\R^d$ satisfy condition (ii).
Then equation \eqref{eq1thpolyhedralriskneutrald} still holds, where
the maximum is taken over the averages of $f$ with respect to all extreme points
of the risk-neutral probabilities on $\xi_1,\cdots, \xi_k$. However, unlike the situation
with condition (i) satisfied, these extreme risk neutral measures may have support not only
on families of size $d+1$ in general positions, but also on families of any size $m+1$,
$m<d$, such that they belong to a subspace of dimension $m$ and form a set of
 general position in  this subspace.
\end{theorem}

\begin{remark}
\label{remnonexchangeminmax}
Notice that $\min$ and $\max$ in \eqref{eq1thpolyhedralriskneutrald} are not interchangeable,
as clearly
\[
\max_i \min_{\ga \in \R^d} [f(\xi_i)-(\xi_i,\ga)]=-\infty.
\]
\end{remark}

Let us now formulate a mirror image of
Theorem \ref{thpolyhedralriskneutrald}, where $\min$ and $\max$ are reversed.
Its proof is almost literally the same as the proof of Theorem \ref{thpolyhedralriskneutrald}.

\begin{theorem}
\label{thriskneutralmir}
Under the assumptions of Theorem \ref{thriskneutralcost}
the expression
\begin{equation}
\label{eq2mainminmaxopmir}
\tilde \Pi[\xi_1, \cdots , \xi_k](f)
=\max_{\ga \in \R^d} \min_{\xi_1,\cdots, \xi_k} [f(\xi_i)+(\xi_i,\ga)]
\end{equation}
can be evaluated by the formula
\begin{equation}
\label{eq2thriskneutralmir}
\Pi[\xi_1, \cdots , \xi_k](f)=\min_{I} \E_I f(\xi),
\end{equation}
where $\min$ is taken over all families $\{\xi_i\}_{i\in I}$, $I\subset \{1,\cdots, k\}$
of size $|I|=d+1$ that satisfy (ii),
and $\E_{I}$ denotes the expectation with respect to the unique risk neutral probability
on $\{\xi_i\}_{i\in I}$. The $\min$ in \eqref{eq2thriskneutralmir} can be also interpreted as taken over
 all extreme points
of the risk-neutral probabilities on $\xi_1,\cdots, \xi_k$.
\end{theorem}

Notice that the r.h.s. of \eqref{eq2thriskneutralmir} is similar to the formula for a
coherent acceptability measure,
see Artzner et al \cite{ArDe} and Roorda, Schumacher and Engwerda \cite{RoSchEn}. However, in the theory of acceptability measures, the collection of measures
with respect to which the minimization is performed, is a subjectively specified. In our model, this collection
is the collection of all extreme points that arises objectively as an evaluation tool for our game-theoretic
problem.

\begin{remark}
Coherent acceptability measures $\phi$ introduced in Artzner et al \cite{ArDe} represent particular cases of nonlinear
averages in the sense of Kolmogorov, see \cite{Kolmog} and Maslov \cite{Mas05}. The distinguished feature that leads to
the representation of $\phi$ as an infimum over probability measures is its super-additivity.
Clearly, postulating sub-additivity, instead of super-additivity, would lead similarly to the representation as a supremum
over probability measures, and hence to the analog of \eqref{eq1thpolyhedralriskneutrald}.
\end{remark}

 \section{A nonlinear extension}
\label{secnonlinext}

Let us discuss a nonlinear extension of the above results. It will be used
 for the analysis of transaction costs.
We start with the simplest additive perturbations, which are sufficient for the static (one-step)
evaluations with transaction costs.

For a finite set of non-vanishing vectors
   $\{\xi_1,\cdots, \xi_k\}$ in $\R^d$, we shall evaluate the expression
\begin{equation}
\label{eq2mainminmaxopcost}
\Pi[\xi_1, \cdots , \xi_k](f,g)=\min_{\ga \in \R^d} \max_{\xi_1,\cdots, \xi_k} [f(\xi_i)-(\xi_i,\ga)+g(\ga)],
\end{equation}
where $g$ is some continuous function. The main example to have in mind is
\[
g(x)=c_1|x^1| +\cdots +c_d|x^d|
\]
with some positive constants $c_i$. We are going to make explicit the (intuitively clear) fact
 that if $g$ is small enough, the $\min$ in \eqref{eq2mainminmaxopcost}
 is attained on the same $\ga$ as when $g=0$.

\begin{theorem}
\label{thriskneutralcost}
 Let $\{\xi_1,\cdots, \xi_k\}$, $k>d$, be a family of vectors in $\R^d$, satisfying the general position conditions (i) and (ii) of Section
   \ref{secunderlyinggt}.

Let $g(\ga)$ be a non-negative Lipshitz continuous function that has well defined derivatives $D_y(x)$
in all point $x$ and in all directions $y$ such that
for any subfamily $\xi_i$, $i\in I\subset \{\xi_1,\cdots, \xi_k\}$, which does not satisfy (ii), one can choose
an $\om$ defining the subspace containing all $\xi_i$, $i\in I$ (i.e. $(\om, \xi_i)>0$ for all $i\in I$) in such a way that
 \begin{equation}
\label{eq1thriskneutralcost}
(\xi_i, \om)> D_{\ga}(\om), \quad i\in I, \, \ga \in R^d.
\end{equation}

Then the minimum in \eqref{eq2mainminmaxopcost} is finite, is attained on some $\ga_0$
and
\begin{equation}
\label{eq2thriskneutralcost}
\Pi[\xi_1, \cdots , \xi_k](f,g)=\max_{I} [\E_I f(\xi)+g(\ga_I)],
\end{equation}
where $\max$ is taken over all families $\{\xi_i\}_{i\in I}$, $I\subset \{1,\cdots, k\}$
of size $|I|=d+1$ that satisfy (ii) (i.e. such that the origin is contained in the interior of
$\Pi [\xi_i, \, i\in I]$),
and $\E_{I}$ denotes the expectation with respect to the unique risk neutral probability
on $\{\xi_i\}_{i\in I}$ (given by Proposition \ref{propriskneutrald}),
 and $\ga_I$ is the corresponding (unique) optimal values.

In particular, if $k=d+1$, then $\ga_0$ is
given by \eqref{mainminmaxopgaarbd},
as in the case of vanishing $g$, and
\begin{equation}
\label{eq3thriskneutralcost}
\Pi[\xi_1, \cdots , \xi_{d+1}](f,g)=\Pi[\xi_1, \cdots , \xi_{d+1}](f)+g (\ga_0)
\end{equation}
 with $\Pi[\xi_1, \cdots , \xi_{d+1}](f)$ from
\eqref{eq73mainminmaxoparbd}.

\end{theorem}

\begin{proof}

Arguing now as in Section \ref{secunderlyinggt}, suppose the $\min$ in \eqref{eq2mainminmaxopcost}
is attained on a vector $\ga_0$ and the corresponding $\max$ is attained precisely on a
subfamily $\xi_i$, $i\in I\subset \{\xi_1,\cdots, \xi_k\}$, so that
\[
f(\xi_i)-(\ga_0,\xi_i)
\]
coincide for all $i\in I$ and
\begin{equation}
\label{eq4thriskneutralcost}
f(\xi_i)-(\ga,\xi_i)+g(\ga)> f(\xi_j)-(\ga,\xi_j)+g(\ga)
\end{equation}
for $j\notin I$ and $\ga=\ga_0$, but this family
does not satisfy (ii).
(This is of course always the case for the subfamilies of the size $|I| <d+1$.)
Let us pick up an $\om $ satisfying \eqref{eq1thriskneutralcost}.
 As for $\ga=\ga_0+\ep\om$,
\[
 f(\xi_i)-(\xi_i,\ga)+g(\ga)
 = f(\xi_i)-(\xi_i,\ga_0)+g(\ga_0)-\ep [(\xi_i,\om)-D_{\om}g(\ga_0)]+o(\ep),
\]
this expression
is less than
\[
f(\xi_i)-(\xi_i,\ga_0)+g(\ga_0)
\]
for small enough $\ep >0$ and all $i\in I$.
But at the same time \eqref{eq4thriskneutralcost} is preserved for small $\ep$
contradicting the minimality of $\ga_0$.
 Hence, if  $\ga_0$ is a minimal point,
 the corresponding $\max$ must be attained on a family satisfying (ii).
 But any such family contains a subfamily with $d+1$ elements only (by the Carath\'eodory theorem).

Do go further, let us assume first that $k=d+1$.
Then a possible value of $\ga_0$ is unique.
Moreover, the minimum exists and is attained on some finite $\ga$, because
\begin{equation}
\label{eq5thriskneutralcost}
  \max_{\xi_1,\cdots,\xi_k} [f(\xi_i)-(\xi_i,\ga) +g(\ga)] \to \infty,
  \end{equation}
as $\ga \to \infty$ (as this holds already for vanishing $g$).
And consequently it is attained on the single possible candidate $\ga_0$.

Let now $k>d+1$ be arbitrary. Using the case $k=d+1$ we can conclude that
\[
\Pi[\xi_1, \cdots , \xi_k](f,g) \ge \max_{I} [\E_I f(\xi)+g(\ga_I)],
\]
and hence the l.h.s is bounded below and \eqref{eq5thriskneutralcost} holds.
Hence the minimum in \eqref{eq2mainminmaxopcost} is attained on some $\ga_0$,
which implies \eqref{eq2thriskneutralcost} due to the characterization of optimal $\ga$
given above.
\end{proof}

\begin{remark}
In case $d=2, k=3$, condition
\eqref{eq1thriskneutralcost} is fulfilled if
for any $i,j$ and $\ga \in \R^2$
 \begin{equation}
\label{remthreepointriskneutralcost}
2|D_{\ga}(\xi_i)|<|\xi_i| \max (|\xi_i|,|\xi_j|) (1+\cos \phi(\xi_i, \xi_j)),
\end{equation}
where by $\phi (x,y)$ we denote the angle between vectors $x,y$.
\end{remark}

Let us turn to the fully nonlinear (in $\ga$) extension of our game-theoretic problem:
to evaluate the minmax expression
\begin{equation}
\label{eq0thriskneutralnonl}
\Pi[\xi_1, \cdots , \xi_k](f)=\min_{\ga \in \R^d} \max_{\xi_1,\cdots, \xi_k} [f(\xi_i, \ga)-(\xi_i,\ga)].
\end{equation}

Let us introduce two characteristics of a system $\{\xi_1, \cdots , \xi_k\}$,
satisfying the general position conditions (i) and (ii) of Section
   \ref{secunderlyinggt}, that measure numerically a spread of the elements of this system
   around the origin.

Let $\ka_1=\ka_1(\xi_1, \cdots , \xi_k)$ be the minimum among the numbers $\ka$ such that
for any subfamily $\xi_i$, $i\in I\subset \{\xi_1,\cdots, \xi_k\}$, which does not satisfy (ii), one can choose
a vector $\om_I \in R^d$ of unit norm such that
 \begin{equation}
\label{eq1thriskneutralnonl}
(\xi_i, \om)\ge \ka, \quad i\in I.
\end{equation}
This $\ka_1$ is clearly positive by conditions (i), (ii).
Let $\ka_2=\ka_2(\xi_1, \cdots , \xi_k)$ be the minimum of the lengths of all perpendiculars from the origin
to the affine hyper-subspaces generated by the end points of any subfamily
containing $d$ vectors.

\begin{theorem}
\label{thriskneutralnonl}
 Let $\{\xi_1,\cdots, \xi_k\}$, $k>d$, be a family of vectors in $\R^d$, satisfying the general position conditions (i) and (ii) of Section
   \ref{secunderlyinggt}.

Let the function $f$ be bounded below and Lipshitz continuous in $\ga$, i.e.
\begin{equation}
\label{eq2thriskneutralnonl}
|f(\xi_i, \ga_1)-f(\xi_i, \ga_2)|\le \ka |\ga_1-\ga_2|
\end{equation}
for all $i$, with a Lipshitz constant that is less than both $\ka_1$ and $\ka_2$:
\begin{equation}
\label{eq21thriskneutralnonl}
\ka < \min (\ka_1, \ka_2).
\end{equation}

Then the minimum in \eqref{eq0thriskneutralnonl} is finite, is attained on some $\ga_0$
and
\begin{equation}
\label{eq3thriskneutralnonl}
\Pi[\xi_1, \cdots , \xi_k](f)=\max_{I} [\E_I f(\xi, \ga_I)],
\end{equation}
where $\max$ is taken over all families $\{\xi_i\}_{i\in I}$, $I\subset \{1,\cdots, k\}$
of size $|I|=d+1$ that satisfy (ii),
$\E_{I}$ denotes the expectation with respect to the unique risk neutral probability
on $\{\xi_i\}_{i\in I}$ (given by Proposition \ref{propriskneutrald}),
 and $\ga_I$ is the corresponding (unique) optimal value, constructed below.

In particular, if $k=d+1$, then $\ga_0$ is
the unique solution of equation \eqref{eqnonlsystforga1} below.
\end{theorem}

\begin{proof}

As in the proof of Theorem \ref{thriskneutralcost}, using now \eqref{eq2thriskneutralnonl}
and \eqref{eq1thriskneutralnonl}, we show that the minimum
cannot be attained on a $\ga$ such that the corresponding maximum is attained only on a
subfamily $\xi_i$, $i\in I\subset \{\xi_1,\cdots, \xi_k\}$, that
does not satisfy (ii).
And again we conclude that if  $\ga_0$ is a minimal point,
 the corresponding $\max$ must be attained on a family satisfying (ii) and containing
$d+1$ elements only.

Let us assume that $k=d+1$.
Then a possible value of $\ga_0$ satisfies the system
  \begin{equation}
\label{eqnonlsystforga}
 (\xi_i-\xi_1, \ga_0)=f(\xi_i, \ga_0)-f(\xi_1, \ga_0), \quad i=2, \cdots, d+1,
\end{equation}
which by \eqref{eq81mainminmaxoparbd} rewrites as
\begin{equation}
 \label{eqnonlsystforga1}
  \ga_0=\E \left[ f(\xi, \ga_0) \frac{\tilde R (\{\hat \xi \})}{D(\{\hat \xi \})}\right],
\end{equation}
where the expectation is with respect to the probability law
\eqref{eq6mainminmaxoparbd}.
This is a fixed point equation. Condition \eqref{eq2thriskneutralnonl},
\eqref{eq21thriskneutralnonl}, the definition of $\ka_2$
and estimate \eqref{eq2estimforPiandga} imply that the mapping on the r.h.s. is a contraction,
and hence equation \eqref{eqnonlsystforga1} has a unique solution $\ga_0$.

Moreover, the minimum in \eqref{eq0thriskneutralnonl} exists and is attained on some finite $\ga$, because
\begin{equation}
\label{eq5thriskneutralnonl}
  \max_{\xi_1,\cdots,\xi_k} [f(\xi_i,\ga)-(\xi_i,\ga)] \to \infty,
  \end{equation}
as $\ga \to \infty$ (as this holds already for vanishing $f$).
And consequently it is attained on the single possible candidate $\ga_0$.

Let now $k>d+1$ be arbitrary. Using the case $k=d+1$ we can conclude that
\[
\Pi[\xi_1, \cdots , \xi_k](f) \ge \max_{I} \E_I f(\xi, \ga_I),
\]
and hence the l.h.s is bounded below and \eqref{eq5thriskneutralnonl} holds.
Hence the minimum in \eqref{eq0thriskneutralnonl} is attained on some $\ga_0$,
which implies \eqref{eq3thriskneutralnonl} due to the characterization of optimal $\ga$
given above.
\end{proof}

In applications to options we need to use Theorem \ref{thriskneutralnonl}
recursively under expanding systems of vectors $\xi$. To this end, we require some estimates
indicating the change of basic coefficients of spread under linear scaling of all co-ordinates.

For a vector $z \in \R^d_+$ with positive coordinates let
\[
\de (z)= \max_i z^i/\min_i z^i.
\]

\begin{prop}
\label{propchangeogkabyscale}
Let a system $\{\xi_1, \cdots , \xi_k\}$ of vectors in $\R^d$
satisfy the general position conditions (i) and (ii) of Section
\ref{secunderlyinggt}. Let $\ka_1, \ka_2$ be the characteristics
of the system $\{\xi_1, \cdots , \xi_k\}$ introduced above and,
for a vector $z\in \R^d$ with positive co-ordinates,
let $\ka_1 (z), \ka_2(z)$ denote the characteristics
of the system $\{z\circ \xi_1, \cdots , z\circ \xi_k\}$. Then
\begin{equation}
\label{eq1propchangeogkabyscale}
\ka_1(z) \ge |z| \ka_1 (d\de (z))^{-1},
\quad \ka_2(z) \ge |z| \ka_2 (\sqrt d\de (z))^{-1}.
\end{equation}
\end{prop}

\begin{proof}
Let us denote by $z^{-1}$, just for this proof,
the vector in $\R^d$ with co-ordinates $z_i^{-1}$.

For a unit vector $\phi=|z^{-1}|^{-1}z^{-1}\circ \om$, we get,
using $(\xi_i, \om)\ge \ka_1$ that
\[
(z\circ \xi_i, \phi)=|z^{-1}|(\xi_i, \om)\ge |z| \ka_1 \frac{1}{|z|\, |z^{-1}|}.
\]
Hence to get the first inequality in \eqref{eq1propchangeogkabyscale} it remains to observe
that
\[
 |z|\, |z^{-1}|=(\sum_{i=1}^d z_i^2 \sum_{i=1}^d z_i^{-2})^{1/2} \le d\de (z).
 \]
 Turning to the proof of the second inequality in \eqref{eq1propchangeogkabyscale} let
 us recall that for any subsystem of $d$ elements that we denote by $u_1,\cdots, u_d$
 the length perpendicular $h$ from the origin to the affine hyperspace generated by the
  end points of vectors $\{u_1,\cdots, u_d\}$ is expressed, by \eqref{eqweightratioforga}, as
 \begin{equation}
 \label{eqweightratioforga1}
  h=\frac{D(u_1,\cdots, u_d)}{\|\tilde R (u_1,\cdots, u_d)\|}.
\end{equation}
From the definition of $D$ as a determinant it follows that
\[
D(\{z\circ u_i\})=\prod_{l=1}^d z_l D(\{u_i\}).
\]
Next, for the $j$th co-ordinate of the rotor $R(\{z\circ u_i\})$ we have
\[
  R^j(\{z\circ u_i\})=\frac{1}{z_j}\prod_{l=1}^d z_l R^j(\{u_i\}),
  \]
  so that
  \[
  \|R(\{z\circ u_i\})\|
  =\prod_{l=1}^d z_l \left(\sum_j \frac{1}{z_j^2} (R^j(\{u_i\}))^2\right)^{1/2}
  \]
  \[
  \le \prod_{l=1}^d z_l  \frac{1}{\min_j z_j} \|R(\{u_i\})\|
  \le \frac{1}{|z|}\prod_{l=1}^d z_l \sqrt d \de (z) \|R(\{u_i\})\|.
  \]
  Hence
  \[
  h(z)= \frac{D(\{z\circ u_i\})}{\|\tilde R (\{z\circ u_i\})\|}
  \ge |z| (\sqrt d \de (z))^{-1},
  \]
  implying the second inequality in \eqref{eq1propchangeogkabyscale}.
\end{proof}

\section{Back to options; properties of solutions: non-expansion and homogeneity}
\label{secbacktooptions}

Let us now calculate the reduced Bellman operator of European colored options
given by \eqref{eqBellmanforop}.
  Changing variables $\xi=(\xi^1,\dots, \xi^J)$ to
$\eta=\xi \circ z$ yields
\begin{equation}
\label{eqBellmanforop1}
(\BC f)(z^1,...,z^J)=\min_{\ga}\max_{\{\eta \in [z^id_i,z^i u_i]\}}
[f(\eta)
-\sum_{i=1}^J\ga^i (\eta^i-\rho z^i)],
\end{equation}
or, by shifting,
\begin{equation}
\label{eqBellmanforop2}
(\BC f)(z^1,...,z^J)=\min_{\ga}\max_{\{\eta \in [z^i(d_i-\rho),z^i (u_i-\rho)]\}}
[\tilde f(\eta)-(\ga, \eta)]
\end{equation}
with $\tilde f(\eta)=f(\eta+\rho z)$. Assuming $f$ is convex (possibly not strictly),
 we find ourselves in the setting of Section \ref{secunderlyinggt}
with $\Pi$ being the rectangular parallelepiped
\[
\Pi_{z,\rho}= \times_{i=1}^J [z^i(d_i-\rho),z^i (u_i-\rho)],
\]
with vertices
\[
\eta_I= \xi_I \circ z -\rho z,
\]
where
\[
\xi_I
= \{ d_i|_{i\in I}, u_j|_{j \notin I}\},
\]
are the vertices of the normalized parallelepiped
\begin{equation}
\label{eqparalforop}
\Pi= \times_{i=1}^J [d_i, u_i],
\end{equation}
parametrized by all subsets (including the empty one) $I\subset \{1,\dots,J\}$.

Since the origin is an internal point of $\Pi$ (because $d_i<\rho <u_i$), condition (ii) of Theorem
\ref{thpolyhedralriskneutrald} is satisfied. Condition (i) is rough in the sense that
 it is fulfilled for an open dense subset of pairs $(d_i,u_i)$. Applying
 Theorem \ref{thpolyhedralriskneutrald} (and Remark \ref{remlinearpartofBelop}) to \eqref{eqBellmanforop2}
 and returning back to $\xi$ yields the following.

\begin{theorem}
\label{thbasicEuroprainbowBel}
If the vertices $\xi_I$ of the parallelepiped $\Pi$ are in general position in the sense that
for any $J$ subsets $I_1,\cdots, I_J$, the vectors $\{\xi_{I_k}-\rho \1\}_{k=1}^J$
are independent in $\R^J$, then
\begin{equation}
\label{eq1thbasicEuroprainbowBel}
(\BC f)(z)=\max_{\{\Om\}} \E_{\Om} f(\xi \circ z), \quad z=(z^1,\cdots, z^J),
\end{equation}
where $\{\Om\}$ is the collection of all subsets $\Om =\xi_{I_1},\cdots,\xi_{I_{J+1}}$
of the set of vertices of $\Pi$, of size $J+1$, such that their convex hull
contains $\rho \1$ as an interior point ($\1$ is the vector with all coordinates 1), and where $\E_{\Om}$ denotes the expectation with respect to the
unique probability law $\{p_I\}$, $\xi_I\in \Om$,
 on the set of vertices of $\Pi$, which is supported on $\Om$
 and is risk neutral with respect to $\rho \1$, that is
\begin{equation}
\label{eq2thbasicEuroprainbowBel}
\sum_{I\subset \{1,\dots,J\}} p_I\xi_I=\rho \1.
\end{equation}
Moreover, if
\[
 f(\xi \circ z )-(\ga_{I_1,\cdots,I_{J+1}}, (\xi -\rho \1)\circ z)
\ge  f(\zeta \circ z)-(\ga_{I_1,\cdots,I_{J+1}}, (\zeta-\rho \1)\circ z)
\]
for all vertices $\xi, \zeta$ such that $\xi \in \Om$ and $\zeta \notin \Om$,
where $\ga_{I_1,\cdots,I_{J+1}}$ is the corresponding optimal value
for the polyhedron $\Pi[\xi_{I_1},\cdots,\xi_{I_{J+1}}]$, then
\begin{equation}
\label{eq3thbasicEuroprainbowBel}
(\BC f)(z^1,...,z^J)= \E_{\Om} f(\xi \circ z).
\end{equation}
\end{theorem}

Risk neutrality now corresponds to its usual meaning in finances, i.e.
\eqref{eq2thbasicEuroprainbowBel} means that all discounted
stock prices are martingales.

Notice that the $\max$ in \eqref{eq1thbasicEuroprainbowBel} is over a finite
number of explicit expressions, which is of course a great achievement as compared
with initial minimax over an infinite set. In particular, it reduces the calculation
of the iterations $\BC^n f$ to the calculation for a control Markov chain.
Let us also stress that the number of eligible $\Om$ in \eqref{eq1thbasicEuroprainbowBel} is
the number of different pyramids (convex polyhedrons with $J+1$ vertices) with vertices
taken from the vertices of $\Pi$ and containing $\rho \1$ as an interior point. Hence this number can be effectively
calculated.

\begin{remark}
Here we used the model of jumps, where each $\xi^i$ can jump
independently in its interval. Thus we used the theory of Section \ref{secunderlyinggt}
only for the case of a polyhedral $\Pi$ being a parallelepiped.
The results of Section \ref{secunderlyinggt} are given in a more general form to allow
more general models of correlated jumps, see end of Section \ref{secpathdep}.
\end{remark}

Let us point our some properties of the operator $\BC$ given by
\eqref{eq1thbasicEuroprainbowBel} that are obvious, but important for practical calculations:
 it is non-expansive:
\[
\| \BC (f_1)- \BC (f_2) \| \le \| f_1-f_2\|,
\]
and homogeneous (both with respect to addition and multiplication):
\[
 \BC (\la + f)=\la + \BC (f), \quad \BC (\la f)=\la \BC (f)
 \]
for a function $f$ and $\la \in \R$ (resp. $\la >0$) for the first (resp second) equation.
Finally, if $f_p$ is a power function, that is
\[
f_p(z)= (z^1)^{i_1} \cdots (z^J)^{i_J},
\]
then $f_p(\xi \circ z)= f_p(\xi)f_p(z)$ implying
\begin{equation}
\label{eqBelonpowers}
(\BC ^n f_p)(z)=((\BC f_p)(\1))^n f_p(z).
\end{equation}
Therefore, power functions are invariant under $\BC$ (up to a multiplication by a constant).
Consequently, if for a payoff $f$ one can find a reasonable approximation by a power function,
that is there exists a power function $f_p$ such that
$\|f-f_p\| \le \ep$, then
\begin{equation}
\label{eqBelonpowersap}
\|\BC ^n f- \la ^n f_p \|=\|f-f_p\| \le \ep, \quad \la = (\BC f_p) (\1),
\end{equation}
so that an approximate calculation of $\BC ^n f$ is reduced to the calculation of one number
$\la $. This implies the following scheme for an approximate evaluation of $\BC$: first
find the best fit to $f$ in terms of functions $\al +f_p$ (where $f_p$ is a power function and $\al$
a constant) and then use \eqref{eqBelonpowersap}.

\section{Sub-modular payoffs}
\label{secsubmodular}

One can get essential reduction in the combinatorics of Theorem \ref{thbasicEuroprainbowBel}
(i.e. in the number of eligible $\Om$) under additional assumptions on the payoff $f$.
The most natural one in the context of options turns out to be the notion of sub-modularity.
A function $f: \R^2_+ \to \R_+$ is called {\it sub-modular}, if the inequality
\[
f(x_1,y_2)+f(x_2,y_1) \ge f(x_1,x_1)+f(y_1,y_2)
\]
holds whenever $x_1 \le y_1$ and $x_2 \le y_2$.
Let us call a function $f: \R^d_+ \to \R_+$ {\it sub-modular} if it is sub-modular with
respect to any two variables.

\begin{remark} If $f$ is twice continuously differentiable, then it
is sub-modular if and only if $\frac{\pa^2 f}{\pa z_i\pa z_j}\leq 0$ for all $i\neq j$.
\end{remark}

As one easily sees, the payoffs of the first three examples of rainbow options,
given in Section \ref{Options}, that is those defined by
\eqref{best risky assets}, \eqref{calls on max}, \eqref{multiple strike options},
are sub-modular. Let us explain,  on the examples of two and three colors $J$,
 how the assumptions of sub-modularity can simplify
Theorem \ref{thbasicEuroprainbowBel}.

Let first $J=2$. The polyhedron \eqref{eqparalforop} is then a rectangle. From sub-modularity
of $f$ it follows that if $\Om$ in Theorem \ref{thbasicEuroprainbowBel} is either
\[
\Om_{12}= \{ (d_1,d_2), (d_1,u_2), (u_1,u_2) \},
\]
or
\[
\Om_{21}= \{ (d_1,d_2), (u_1,d_2), (u_1,u_2) \},
\]
then
$(f,\xi)-(\ga_0,\xi)$ coincide for all vertices $\xi$ of $\Pi$.
Hence $\Om_{12}$ and $\Om_{21}$ can be discarded in Theorem \ref{thbasicEuroprainbowBel}, i.e the maximum
is always achieved either on
\[
\Om_d= \{ (d_1,d_2), (d_1,u_2), (u_1,d_2) \},
\]
or on
\[
\Om_u= \{ (d_1,u_2), (u_1, d_2), (u_1,u_2) \}.
\]
But the interiors of the triangle formed by $\Om_u$ and $\Om_d$ do not intersect, so that
each point of $\Pi$ (in general position) lies only in one of them (and this position does not depend
 any more on $f$). Hence, depending on the position of $\rho \1$ in $\Pi$, the expression
\eqref{eq1thbasicEuroprainbowBel} reduces either to
$\E_{\Om_u}$ or to $\E_{\Om_d}$. This yields the following result (obtained in
\cite{Ko98}).

\begin{theorem}
\label{twostocks}
 Let $J=2$ and $f$ be convex sub-modular.
Denote
\begin{equation}
\kappa =\frac{(u_1u_2-d_1d_2)-\rho
(u_1-d_1+u_2-d_2)}{(u_1-d_1)(u_2-d_2)}
=1-\frac{\rho-d_1}{u_1-d_1}-\frac{\rho-d_2}{u_2-d_2}.
\end{equation}
If $\kappa \geq 0$, then $(\mathcal{B}f)(z_1,z_2)$ equals
\begin{equation}
 \frac{\rho-d_1}{u_1-d_1}f(u_1z_1,d_2z_2)+\frac{\rho
-d_2}{u_2-d_2}f(d_1z_1,u_2z_2)+\kappa f(d_1z_1,d_2z_2),
\end{equation}
and the corresponding optimal strategies are
\[
\ga^1=\frac{f(u_1z_1,d_2z_2)-f(d_1z_1,d_2z_2)}{z_1(u_1-d_1)},
\quad
\ga^2=\frac{f(d_1z_1,u_2z_2)-f(d_1z_1,d_2z_2)}{z_2(u_2-d_2)}.
\]
If $\kappa \leq 0$, the $(\mathcal{B}f)(z_1,z_2)$ equals
\begin{equation}
 \frac{u_1-\rho}{u_1-d_1}f(d_1z_1,u_2z_2)+\frac{u_2-\rho}{u_2-d_2}f(u_1z_1,d_2z_2)+|\kappa| f(u_1z_1,u_2z_2),
\end{equation}
and
\[
\ga^1=\frac{f(u_1z_1,u_2z_2)-f(d_1z_1,u_2z_2)}{z_1(u_1-d_1)},
\quad
\ga^2=\frac{f(u_1z_1,u_2z_2)-f(u_1z_1,d_2z_2)}{z_2(u_2-d_2)}.
\]
\end{theorem}

Clearly the linear operator $\mathcal{B}$ preserves the set of
convex sub-modular functions. Hence one can use
this formula recursively to obtain all powers of
$\mathcal{B}$ in a closed form. For instance in case $\kappa =0$ one obtains for the
hedge price the following two-color extension of the classical Cox-Ross-Rubinstein formula:
\begin{equation}
\label{crr2}
\mathcal{B}^n f(S_0^1,S_0^2)
=\rho^{-n}\sum_{k=0}^n C^k_n
\left(\frac{\rho-d_1}{u_1-d_1}\right)^k
\left(\frac{\rho-d_2}{u_2-d_2}\right)^{n-k} f(u_1^kd_1^{n-k}S_0^1,
d_2^ku_2^{n-k}S_0^2).
\end{equation}

Now let $J=3$. Then polyhedron \eqref{eqparalforop} is a parallelepiped in $\R^3$.
From sub-modularity projected on the first two co-ordinates we conclude that whenever
the vertices $(d_1,d_2,d_3)$ and $(u_1,u_2,d_3)$ are in $\Om$, then
\[
f(\xi)-(\ga_0,\xi)
\]
should coincide for $\xi$ being $(d_1,d_2,d_3)$, $(u_1,u_2,d_3)$,
$(u_1,d_2,d_3)$, $(d_1,u_2,d_3)$. In other word, the pair $(d_1,d_2,d_3)$, $(u_1,u_2,d_3)$
can be always substituted by the pair $(u_1,d_2,d_3)$, $(d_1,u_2,d_3)$. Consequently,
those $\Om$ containing the pair $(d_1,d_2,d_3), (u_1,u_2,d_3)$ are superfluous, they can
be discarded from the possible $\Om$ competing in formula \eqref{eq1thbasicEuroprainbowBel}.
Similarly, we can discard all those $\Om$ containing six pairs, three of which
containing $(d_1,d_2,d_3)$ and one among $(d_1,d_2,u_3)$, $(d_1,u_2,d_3)$, $(u_1,d_2,d_3)$,
and other three containing $(u_1,u_2,u_3)$  and one among $(d_1,u_2,u_3)$, $(u_1,u_2,d_3)$, $(u_1,d_2,u_3)$.

These considerations reduce dramatically the number of eligible $\Om$.
In particular, if $\rho\1$ lies in the tetrahedron $\Om_d$ formed by the vertices
$(d_1,d_2,d_3)$, $(d_1,d_2,u_3)$, $(d_1,u_2,d_3)$, $(u_1,d_2,d_3)$,
then the only eligible $\Om$ is $\Om_d$. If $\rho\1$ lies in the tetrahedron $\Om_u$ formed by the vertices
$(u_1,u_2,u_3)$, $(d_1,u_2,u_3)$, $(u_1,u_2,d_3)$, $(u_1,d_2,u_3)$,
then the only eligible $\Om$ is $\Om_u$.
Formally these cases are easily seen to be distinguished by the inequalities
$\al_{123} >0$ and $\al_{123} <-1$ respectively, where
\begin{equation}
\label{eqalforthreeop}
\al_{123}=\left( 1-\frac{u_1-\rho}{u_1-d_1}-\frac{u_2-\rho}{u_2-d_2}
-\frac{u_3-\rho}{u_3-d_3}\right).
\end{equation}

This yields the following result (by continuity we are able to write $\al_{123} \ge 0$ instead of a strict inequality), where we use the following notation:
for a set $I\subset \left\{ 1,2,...,J\right\} $,
$f_{I}(z)$ is $f(\xi ^1z_1, \cdots,\xi^Jz_J)$ with $\xi^i=d_i$ for $i\in I$
and $\xi_i=u_i$ for $i\notin I$.

\begin{theorem}
\label{threestocks1}
Let $J=3$ and $f$ be continuous convex and
sub-modular.

(i) If $\al_{123} \ge 0$, then

\begin{equation}
(\mathcal{B}f)(z)
=\frac{1}{\rho}\left[\al_{123}f_{\emptyset }(z)
+\frac{u_1-\rho}{u_1-d_1}f_{\{1\}}(z)
+\frac{u_2-\rho}{u_2-d_2}f_{\{2\}}(z)
+\frac{u_3-r}{u_3-d_3} f_{\{3\}}(z)\right].
\end{equation}

(ii) If $\alpha _{123}\leq -1$, then

\begin{equation}
(\mathcal{B}f)(z)
=\frac{1}{\rho}\left[-(\al_{123}+1)f_{\{1,2,3\}}(z)
+\frac{\rho-d_1}{u_1-d_1}f_{\{2,3\}}(z)
+\frac{\rho-d_2}{u_2-d_2} f_{\{1,3\}}(z)
+\frac{\rho-d_3}{u_3-d_3}f_{\{1,2\}}(z)\right].
\end{equation}
\end{theorem}

Hence in these cases, our
$\mathcal{B}$ again reduces to a linear form, allowing for a straightforward calculation
of its iterations, as in case $J=2$ above.

Suppose now that $\rho\1$ lies neither in the tetrahedron $\Om_d$, nor in $\Om_u$ (i.e. neither of the conditions
of Theorem \ref{threestocks1} are satisfied). From the above reductions of possible $\Om$, it follows
that in that case one can discard all $\Om$ containing either $(d_1,d_2,d_3)$ or $(u_1,u_2,u_3)$.
Hence only six vertices are left for eligible $\Om$. From the consideration of general position we
further deduce that altogether only six $\Om$ are possible, namely the three tetrahedrons containing
the vertices $(d_1,d_2,u_3)$, $(d_1,u_2,d_3)$, $(u_1,d_2,d_3)$ and one vertex from
$(d_1,u_2,u_3)$, $(u_1,u_2,d_3)$, $(u_1,d_2,u_3)$, and symmetrically the three tetrahedrons containing
the vertices $(d_1,u_2,u_3)$, $(u_1,u_2,d_3)$, $(u_1,d_2,u_3)$ and one vertex from
$(d_1,d_2,u_3)$, $(d_1,u_2,d_3)$, $(u_1,d_2,d_3)$. However, any particular point in general position
 belongs to only three out of these six leaving in formula \eqref{eq1thbasicEuroprainbowBel}
 the $\max$ over three possibilities only. The particular choice of these three tetrahedrons depends on the
 coefficients
 \begin{equation}
 \label{eqaliforthreeop}
\begin{array}{c}
\alpha _{12}
=\left( 1-\frac{u_{1}-r}{u_{1}-d_{1}}-\frac{u_{2}-r}{u_{2}-d_{2}}\right) \\
\alpha _{13}
=\left( 1-\frac{u_{1}-r}{u_{1}-d_{1}}-\frac{u_{3}-r}{u_{3}-d_{3}}\right) \\
\alpha _{23}
=\left( 1-\frac{u_{2}-r}{u_{2}-d_{2}}-\frac{u_{3}-r}{u_{3}-d_{3}}\right),
\end{array}
\end{equation}
and leads to the following result obtained in Hucki and Kolokoltsov \cite{HuKo} (though with much more elaborate proof than here).

\begin{theorem}
\label{threestocks2}
 Let again $f$ be convex and
sub-modular, but now $0 > \alpha _{123} > -1$.

\begin{tabular}{l}
\textit{(i) If }$\alpha _{12}\geq 0$\textit{, }$\alpha _{13}\geq
0$\textit{\
and }$\alpha _{23}\geq 0,$\textit{\ then} \\
$(\mathcal{B}f)(\mathbf{z})=\frac{1}{r}\max \left\{
\begin{array}{c}
\left( -\alpha _{123}\right) \allowbreak
f_{\{1,2\}}(\mathbf{z})+\allowbreak
\alpha _{13}f_{\{2\}}(\mathbf{z})+\alpha _{23}f_{\{1\}}(\mathbf{z})+\frac{%
u_{3}-r}{u_{3}-d_{3}}f_{\{3\}}(\mathbf{z}) \\
\left( -\alpha _{123}\right) f_{\{1,3\}}(\mathbf{z})+\alpha _{12}f_{\{3\}}(%
\mathbf{z})+\alpha _{23}f_{\{1\}}(\mathbf{z})+\frac{u_{2}-r}{u_{2}-d_{2}}%
f_{\{2\}}(\mathbf{z}) \\
\left( -\alpha _{123}\right) \allowbreak
f_{\{2,3\}}(\mathbf{z})+\alpha
_{12}\allowbreak f_{\{3\}}(\mathbf{z})+\alpha _{13}f_{\{2\}}(\mathbf{z}%
)+\allowbreak \frac{u_{1}-r}{u_{1}-d_{1}}f_{\{1\}}(\mathbf{z})%
\end{array}
\right\} \mathit{,}$
\end{tabular}

\begin{tabular}{l}
\textit{(ii) If }$\alpha _{ij}\leq 0$\textit{, }$\alpha _{jk}\geq 0$\textit{%
\ and }$\alpha _{ik}\geq 0$\textit{, } \\
\textit{where \{}$i,j,k\}$ is an arbitrary permutation of the set $\{1,2,3\}$%
, then \\
$(\mathcal{B}f)(\mathbf{z})=\frac{1}{r}\max \left\{
\begin{array}{c}
\left( -\alpha _{ijk}\right) \allowbreak
f_{\{i,j\}}(\mathbf{z})+\allowbreak
\alpha _{ik}f_{\{j\}}(\mathbf{z})+\alpha _{jk}f_{\{i\}}(\mathbf{z})+\frac{%
u_{k}-r}{u_{k}-d_{k}}f_{\{k\}}(\mathbf{z}) \\
\alpha _{jk}\allowbreak f_{\{i\}}(\mathbf{z})+(-\alpha
_{ij})\allowbreak
f_{\{i,j\}}(\mathbf{z})+\frac{u_{k}-r}{u_{k-}d_{k}}\allowbreak f_{\{i,k\}}(%
\mathbf{z})-\frac{d_{i}-r}{u_{i-}d_{i}}\allowbreak f_{\{j\}}(\mathbf{z}) \\
\alpha _{ik}\allowbreak f_{\{j\}}(\mathbf{z})+(-\alpha
_{ij})\allowbreak
f_{\{i,j\}}(\mathbf{z})+\frac{u_{k}-r}{u_{k-}d_{k}}\allowbreak f_{\{j,k\}}(%
\mathbf{z})-\frac{d_{j}-r}{u_{j-}d_{j}}\allowbreak f_{\{i\}}(\mathbf{z})%
\end{array}
\right\} \mathit{,}$%
\end{tabular}

\begin{tabular}{l}
\textit{(iii) If }$\alpha _{ij}\geq 0$\textit{,}$\alpha _{jk}\leq 0$\textit{%
\ and }$\alpha _{ik}\leq 0$\textit{,} \\
\textit{where \{}$i,j,k\}$ is an arbitrary permutation of the set $\{1,2,3\}$%
, then \\
$(\mathcal{B}f)(\mathbf{z})=\frac{1}{r}\max \left\{
\begin{array}{c}
\alpha _{ij}\allowbreak f_{\{k\}}(\mathbf{z})+(-\alpha
_{jk})\allowbreak
f_{\{j,k\}}(\mathbf{z})+\frac{u_{i}-r}{u_{i-}d_{i}}\allowbreak f_{\{i,k\}}(%
\mathbf{z})-\frac{d_{k}-r}{u_{k-}d_{k}}\allowbreak f_{\{j\}}(\mathbf{z}) \\
\alpha _{ij}\allowbreak f_{\{k\}}(\mathbf{z})+(-\alpha
_{ik})\allowbreak
f_{\{i,k\}}(\mathbf{z})+\frac{u_{j}-r}{u_{j-}d_{j}}\allowbreak f_{\{j,k\}}(%
\mathbf{z})-\frac{d_{k}-r}{u_{k-}d_{k}}\allowbreak f_{\{i\}}(\mathbf{z}) \\
(\alpha _{123}+1)f_{\{k\}}(\mathbf{z})-\alpha _{jk}f_{\{j,k\}}(\mathbf{z}%
)-\alpha _{ik}f_{\{i,k\}}(\mathbf{z})-\frac{d_{k}-r}{u_{k}-d_{k}}f_{\{i,j\}}(%
\mathbf{z})
\end{array}
\right\} .$
\end{tabular}
\end{theorem}

One has to stress here that the application of Theorem \ref{threestocks2} is rather limited:
as $\mathcal{B}$ is not reduced to a linear form, it is not clear how to use it for the iterations
 of $\mathcal{B}$, because the sub-modularity does not seem to be preserved under such $\mathcal{B}$.

\section{Transaction costs}
\label{sectransco}

Let us now extend the model of Section \ref{Options} to include possible transaction costs.
They can depend on transactions in various way. The simplest for the analysis are the so called
{\it fixed transaction costs} that equal to a fixed fraction $(1-\be)$ (with $\be$ a small constant)
 of the entire portfolio. Hence for fixed costs, equation \eqref{eqnewcap1} changes to

\begin{equation}
\label{eqnewcap1fixedtr}
X_m=\be \sum_{j=1}^J\ga _m^j \xi _m^j S_{m-1}^j+\rho
(X_{m-1}-\sum_{j=1}^J \ga _m^j S_{m-1}^j).
\end{equation}

As one easily sees, including fixed costs can be dealt with by re-scaling $\rho$, thus bringing nothing new
to the analysis.

In more advanced models, transaction costs depend on the amount of transactions (bought and sold stocks)
in each moment of time, i.e. are given by
some function
\[
g(\ga_m-\ga_{m-1}, S_{m-1}),
\]
and are payed at time, when the investor changes $\ga_{m-1}$ to $\ga_m$.
In particular, the basic example present the so called {\it proportional transaction costs}, where
\[
g(\ga_m-\ga_{m-1}, S_{m-1})=
\be \sum_{j=1}^J |\ga^j_m-\ga^j_{m-1}|S_{m-1}^j
\]
(again with a fixed $\be >0$).
We shall assume only that $g$ has the following Lipshitz property:
\begin{equation}
\label{eqLipcosts}
|g(\ga_1,z)-g(\ga_2, z)| \le \be |z| |\ga_1-\ga_2|
\end{equation}
with a fixed $\be >0$.

To deal with transaction costs, it is convenient to extend the state space of our game, considering the states
 that are characterized, at time $m-1$, by $2J+1$ numbers
\[
X_{m-1}, S_{m-1}^j, v_{m-1}=\ga_{m-1}^j, \quad j=1,\cdots, J.
\]
When, at time $m-1$, the investor chooses his new control parameters $\ga_m$,
the new state at time $m$ becomes
\[
X_m, \quad S_m^j= \xi^j_m S_{m-1}^j, \quad v_m=\ga_m^j, \quad j=1,\cdots, J,
\]
where the value of the portfolio is
\begin{equation}
\label{eqnewcap1proportr}
X_m=\sum_{j=1}^J \ga _m^j \xi _m^j S_{m-1}^j +\rho
(X_{m-1}-\sum_{j=1}^J \ga _m^j S_{m-1}^j)
 - g(\ga_m-v_{m-1},S_{m-1}).
\end{equation}
The corresponding reduced Bellman operator from Section \ref{Options}
takes the form
\begin{equation}
\label{eqBellmanforoptr}
(\BC f)(z,v)=\min_{\ga}\max_{\xi}
[f(\xi \circ z, \ga)
-(\ga, \xi \circ z-\rho z)+g(\ga-v,z)],
\end{equation}
where $z,v\in \R^J$,
or, changing variables $\xi=(\xi^1,\dots, \xi^J)$ to
$\eta=\xi \circ z$ and shifting,
\begin{equation}
\label{eqBellmanforop2tr}
(\BC f)(z,v)=\min_{\ga}\max_{\{\eta^j \in [z^j(d_j-\rho),z^j (u_j-\rho)]\}}
[f(\eta+\rho z, \ga)-(\ga, \eta)+g(\ga-v,z)].
\end{equation}
On the last step, the function $f$ does not depend on $\ga$, so that Theorem \ref{thriskneutralcost}
can be used for the calculation. But for the next steps Theorem \ref{thriskneutralnonl}
is required.

For its recursive use, let us assume that
\[
|f(z,v_1)-f(z,v_2)| \le \al |z| |v_1-v_2|,
\]
and $\al$ is small enough so that the requirements of Theorem \ref{thriskneutralnonl}
are satisfied for the r.h.s. of \eqref{eqBellmanforop2tr}.
By Theorem \ref{thriskneutralnonl},
\begin{equation}
\label{eqBellmanforop3tr}
(\BC f)(z,v)= \max_{\Om}\E_{\Om} [f(\xi \circ z, \ga_{\Om})+g(\ga_{\Om}-v,z)].
\end{equation}
Notice that since the term with $v$ enters additively,
they cancel from the equations for $\ga_{\Om}$, so that the values of $\ga_{\Om}$
do not depend on $v$. Consequently,
\begin{equation}
\label{eqBellmanforoptrLip}
|(\BC f)(z,v_1)-(\BC f)(z,v_2)|
\le \max_{\Om}\E_{\Om} [g(\ga_{\Om}-v_1, z)-g(\ga_{\Om}-v_2),z)]
\le \be |z| |v_1-v_2|.
\end{equation}

Hence, if at all steps the application of Theorem \ref{thriskneutralnonl} is allowed,
then $(\BC^k f)(z,v)$ remains Lipshitz in $v$ with the Lipshitz constant $\be |z|$
(the last step function does not depend on $v$ and hence trivially satisfies this condition).

Let $\ka_1$, $\ka_2$ be the characteristics, defined before Theorem \ref{thriskneutralnonl},
of the set of vertices $\xi_I-\rho \1$ of the parallelepiped $\times_{j=1}^J[d_j,u_j]-\rho \1$.
By Proposition \ref{propchangeogkabyscale}, the corresponding characteristics
$\ka_1(z)$, $\ka_2(z)$ of the set of vertices of the scaled parallelepiped
\[
\times_{j=1}^J[z^jd_j,z^ju_j]-\rho z
\]
have the lower bounds
\[
\ka_i(z) \ge |z| \ka_i \frac{1}{d\de (z)}, \quad i=1,2.
\]
As in each step of our process the coordinates of $z$ are multiplied by $d_j$ or $u_j$,
 the corresponding maximum $\de_n(z)$ of the
 $\de$ of all $z$ that can occur in the $n$-step process equals
 \begin{equation}
 \label{eq1thtranscost}
 \de_n(z)=\de (z) \left(\frac{\max_j u_j}{\min_j d_j}\right)^n.
 \end{equation}
 Thus we arrive at the following result.

 \begin{theorem}
 \label{thtranscost}
 Suppose $\be$ from \eqref{eqLipcosts} satisfies the estimate
 \[
 \be < \min( \ka_1, \ka _2) \frac{1}{d\de_n(z)},
 \]
 where $\de_n(z)$ is given by \eqref{eq1thtranscost}. Then the hedge
  price of a derivative security specified by a final payoff $f$ and with
  transaction costs specified above is given by \eqref{hedgeprice}, where $\BC$
  is given by  \eqref{eqBellmanforoptr}. Moreover, at each step, $\BC$
  can be evaluated by Theorem \ref{thriskneutralnonl}, i.e. by \eqref{eqBellmanforop3tr},
  reducing the calculations to finding a maximum over a finite set.
 \end{theorem}

 Of course, for larger $\be$, further adjustments of Theorem \ref{thriskneutralnonl} are required.

\section{Rainbow American options and real options}
\label{secamerandreal}

In the world of American options, when an option can be exercised at any time,
the operator
$\mathbf{B}G(X,S^{1},...,S^{J})$ from \eqref{eqBellmanforopnonred} changes to

\[
\mathbf{B}G(X,S^1,\cdots ,S^J)
\]
\begin{equation}
\label{eqBellmanforopnonredAm}
 =\max_{\ga}
 \min \left[G(X,S^1,\cdots ,S^J), \frac{1}{\rho} \min_{\xi}
 G(\rho X+ \sum_{i=1}^{J}\ga^i\xi^iS^{i}
  -\rho \sum_{i=1}^{J}\ga^iS^i,\xi^1S^1, \cdots,\xi^JS^J)\right],
\end{equation}
so that the corresponding reduced operator takes the form
\begin{equation}
\label{eqBellmanforopAm}
(\BC f)(z^{1},...,z^{J})
=\min_{\ga}\max \left[ \rho f(\rho z), \max_{\xi}
[f(\xi ^1 z^1,\xi^2 z^2, \cdots,\xi^J z^J)
-\sum_{i=1}^J\ga^i z^i(\xi^i-\rho)]\right],
\end{equation}
or equivalently
\begin{equation}
\label{eqBellmanforopAm1}
(\BC f)(z^{1},...,z^{J})
=\max \left[\rho f(\rho z),
\min_{\ga}\max_{\xi}
[f(\xi ^1 z^1,\xi^2 z^2, \cdots,\xi^J z^J)
-\sum_{i=1}^J\ga^i z^i(\xi^i-\rho)]
\right].
\end{equation}

Consequently, in this case the main formula
\eqref{eq1thbasicEuroprainbowBel} of Theorem \ref{thbasicEuroprainbowBel}
becomes
\begin{equation}
\label{eq1thbasicEuroprainbowBelAm}
(\BC f)(z^1,...,z^J)=\max\left[ \rho f(\rho z), \max_{\{\Om\}} \E_{\Om} f(\xi \circ z)
\right],
\end{equation}
which is of course not an essential increase in complexity.
The hedge price for the $n$-step model is again given by \eqref{hedgeprice}.

Similar problems arise in the study of real options.
We refer to Dixit and Pindyck \cite{DiPi} for a general background and to
Bensoussan et al \cite{Bens10} for more recent mathematical
results. A typical real option problem can be formulated as follows. Given $J$ instruments
(commodities, assets, etc), the value of the investment in some project at time $m$ is
supposed to be given by certain functions $f_m (S_m^1, \cdots, S_m^J)$ depending on the prices
of these instruments at time $m$. The problem is to evaluate the price (at the initial time $0$)
of the option to invest in this project that can be exercised at any time during a given
time-interval $[0,T]$. Such a price is important, since to keep the option open a firm needs
to pay ceratin costs (say, keep ready required facilities or invest in research). We have
formulated the problem in a way that makes it an example of the general evaluation of an
American rainbow option, discussed above, at least when underlying instruments are tradable
on a market. For practical implementation, one only has to keep in mind that the risk free
rates appropriate for the evaluation of real options are usually not the available bank accounts
used in the analysis of financial options, but rather the growth rates of the corresponding branch
of industry. These rates are usually estimated via the CAPM (capital asset pricing model),
see again \cite{DiPi}.

\section{Path dependence and other modifications}
\label{secpathdep}

The Theory of Section \ref{secbacktooptions} is rough, in the sense that it can be easily modified to
accommodate various additional mechanisms of price generations. We have already considered transaction costs
and American options. Here we shall discuss other three modifications: path dependent payoffs, time depending jumps
(including variable volatility) and nonlinear jump formations. For simplicity, we shall discuss these extensions separately,
but any their combinations (including transaction costs and American versions) can be easily dealt with.

Let us start with path dependent payoffs. That is, we generalize the setting of Section \ref{Options} by
making the payoff $f$ at time $m$ to depend on the whole history of the price evolutions, i.e.
being defined by a function $f(S_0,S_1,\cdots, S_m)$, $S_i=(S_i^1,\cdots, S_i^J)$, on $\R^{J(m+1)}$.
The state of the game at time $m$ must be now specified by $(m+1)J+1$ numbers
\[
X_m, \,\, S_i=(S_i^1,\cdots, S_i^J), \quad i=0,\cdots,m.
\]
The final payoff in the $n$-step game is now $G=X-f(S_0,\cdots,S_n)$ and at the pre ultimate period $n-1$
(when $S_0,\cdots,S_{n-1}$ are known) payoff equals

\[
\mathbf{B}G(X,S_0,\cdots,S_{n-1})
 =X - \frac{1}{\rho}\min_{\ga}\max_{\xi}
 [f(S_0,\cdots, S_{n-1}, \xi \circ S_{n-1})
-(\ga, S_{n-1} \circ (\xi-\rho \1))]
\]
\[
= X -\frac{1}{\rho}(\BC_{n-1}f)(S_0, \cdots ,S_{n-1}),
\]
where the modified {\it reduced Bellman operators} are now defined as
\begin{equation}
\label{eqBellmanforoppathdep}
(\BC_{m-1} f)(z_0,\cdots,z_{m-1})
 =\min_{\ga}\max_{\{\xi^j\in [d_j,u_j]\}}
[f(z_0,\cdots,z_{m-1}, \xi \circ z_{m-1})
-(\ga, \xi \circ z-\rho z)].
\end{equation}

Consequently, by dynamic programming, the guaranteed payoff at the initial moment of time equals
\[
X-\frac{1}{\rho^n} \BC_0 (\BC_1 \cdots (\BC_{n-1}f) \cdots ),
\]
and hence the hedging price becomes
\begin{equation}
 \label{hedgepricepathdep}
H^n=\frac{1}{\rho^{n}}\BC_0 (\BC_1 \cdots (\BC_{n-1}f) \cdots ).
\end{equation}

No essential changes are required if possible sizes of jumps are time dependent.
Only the operators $\BC_{m-1}$ from \eqref{eqBellmanforoppathdep} have to be generalized to
\begin{equation}
\label{eqBellmanforoppathdeptimedep}
(\BC_{m-1} f)(z_0,\cdots,z_{m-1})
 =\min_{\ga}\max_{\{\xi^j\in [d_j^m,u_j^m]\}}
[f(z_0,\cdots,z_{m-1}, \xi \circ z_{m-1})
-(\ga, \xi \circ z-\rho z)],
\end{equation}
where the pairs $(d_j^m,u_j^m)$, $j=1,\cdots, J$, $m=1,\cdots, n$ specify the model.

Let us turn to nonlinear jump patterns. Generalizing the setting of Section \ref{Options}
let us assume, instead of the stock price changing model $S_{m+1}=\xi \circ S_m$, that we are given
$k$ transformations $g_i: \R^J\to \R^J$, $i=1,\cdots, k$, which give rise naturally to two models
of price dynamics: either

(i) at time $m+1$ the price $S_{m+1}$ belongs to the closure of the convex hull of the set $\{g_i(S_m)\}$,
$i=1,\cdots, k$ (interval model), or

(ii)   $S_{m+1}$ is one of the points $\{g_i(S_m)\}$,
$i=1,\cdots, k$.

Since the first model can be approximated by the second one (by possibly increasing the number of transformations
$g_i$), we shall work with the second model.

\begin{remark} Notice that maximizing a function over a convex polyhedron is equivalent
to its maximization over the edges of this polyhedron. Hence, for convex payoffs the two models above
are fully equivalent. However, on the one hand, not all reasonable payoffs are convex, and on the other hand,
when it comes to minimization (which one needs, say, for lower prices, see Section \ref{secupperlowerprice}),
the situation becomes rather different.
\end{remark}

Assuming for simplicity that possible jump sizes are time independent and the payoffs depend only on the end-value
of a path, the reduced Bellman operator
\eqref{eqBellmanforop}
becomes
\begin{equation}
\label{eqBellmanforopnonljump}
(\BC f)(z)=\min_{\ga}\max_{i\in\{1,\cdots,k\}}
[f(g_i(z))-(\ga, g_i(z)-\rho z)],
\quad z=(z^1,...,z^J),
\end{equation}
or equivalently
\begin{equation}
\label{eqBellmanforopnonljump1}
(\BC f)(z)=\min_{\ga}\max_{\eta_i \in \{g_i(z)\}, i=1,\cdots,k}
[f(\eta_i+\rho z)
-(\ga, \eta_i)].
\end{equation}

The hedge price is still given by \eqref{hedgeprice}
and operator \eqref{eqBellmanforopnonljump} is calculated by
Theorem \ref{thriskneutrald}.

It is worth noting that if $k=d+1$ and $\{g_i (z)\}$ form a collection of vectors in a general position,
the corresponding risk-neutral probability is unique. Consequently our hedge price becomes fair in the
sense of 'no arbitrage' in the strongest sense: no positive surplus is possible for all paths
of the stock price evolutions (if the hedge strategy is followed). In particular, the evaluation
of hedge strategies can be carried out in the framework of the standard approach to option pricing.
Namely, choosing as an initial (real world) probability on jumps an arbitrary measure with a full support,
one concludes that there exists a unique risk neutral equivalent martingale measure, explicitly
defined via formula \eqref{eq6mainminmaxoparbd}, and the hedge price calculated by the iterations of operator
\eqref{eqBellmanforopnonljump1} coincides with the standard risk-neutral evaluation of derivative prices
in complete markets.

\section{Upper and lower values; intrinsic risk}
\label{secupperlowerprice}

The celebrated non-arbitrage property of the hedge price of an option in CRR or Black-Scholes
models means that a.s., with respect to the initial probability distribution on paths, the investor
cannot get an additional surplus when adhering to the hedge strategy that guarantees that there could be no loss.
In our setting, even though our formula in case $J=1$ coincides with the CRR formula, we do not assume
any initial law on paths, so that the notion of 'no arbitrage' is not specified either.

It is a characteristic feature of our models that picking up an a priori probability law
with support on all paths leads to an incomplete market, that is to the existence of infinitely many
equivalent martingale measures, which fail to identify a fair price in a unique consistent way.
Notice that our extreme points are absolutely continuous, but usually not equivalent to
a measure with a full support.

\begin{remark}
The only cases of a complete market among the models discussed above are those
mentioned at the end of Section \ref{secpathdep}, that is the models
with precisely $J+1$ eligible jumps of a stock price vector $S$ in each period.
\end{remark}

For the analysis of incomplete markets is of course natural to look for some subjective criteria
to specify a price. Lots of work of different authors were devoted to
this endeavor. Our approach is to search for objective bounds (price intervals), which are given by our hedging
strategies.

\begin{remark}
Apart from supplying the lower price (as below), one can also argue about the reasonability
of our main hedging price noting that a chance for a possible surplus can be (and actually is indeed)
compensated by inevitable inaccuracy of a model, as well as by transaction costs (if they
are not taken into account properly). Moreover, this price satisfies the so called
'no strictly acceptable opportunities' (NSAO) condition suggested in Carr, Geman and Madan \cite{Carr01}.
\end{remark}

For completeness, let us recall the general definitions of lower and upper prices,
in the game theoretic approach to probability, given in Shafer and Vovk \cite{ShV}.
Assume a process (a sequence of real numbers of a fixed length, say $n$, specifying the evolution
of the capital of an investor) is specified by alternating
moves of two players, an investor and the Nature, with complete information (all eligible moves
of each player and their results are known to each player at any time, and the moves become publicly known
at the moment when decision is made). Let us denote by $X^{\al}_{\ga}(\xi)$ the last number of the resulting sequence, starting with an initial value $\al$
and obtained by applying the strategy $\ga$ (of the investor) and $\xi$ (of the Nature).
By a random variable we mean just a function on the set of all
possible paths. The upper value (or the upper expectation) $\overline{\E} f$ of a random variable $f$
is defined as the minimal capital of the investor such that he/she has a strategy that guarantees that at the
final moment of time, his capital is enough to buy $f$, i.e.
\[
\overline{\E} f=\inf \{ \al: \, \exists \ga: \, \forall \xi,  \, X^{\al}_{\ga}(\xi)-f(\xi) \ge 0 \}.
\]
Dually,
the lower value (or the lower expectation) $\overline{\E} f$ of a random variable $f$
is defined as the maximum capital of the investor such that he/she has a strategy that guarantees that at the
final moment of time, his capital is enough to sell $f$, i.e.
\[
\underline{\E} f=\sup \{ \al: \, \exists \ga: \, \forall \xi,  \, X^{\al}_{\ga}(\xi)+f(\xi) \ge 0 \}.
\]
One says that the prices are consistent if $\overline{\E} f \ge \underline{\E} f$. If these prices coincide, we
are in a kind of abstract analog of a complete market. In the general case, upper and lower prices
are also referred to as a seller and buyer prices respectively.

It is seen now that in this terminology our hedging price for a derivative security
is the upper (or seller) price.
The lower price can be defined similarly. Namely, in the setting of Section \ref{Options},
lower price is given by
\[
\frac{1}{\rho^{n}}(\BC^{n}_{low}f)(S_0^1, \cdots , S_0^J),
\]
where
\begin{equation}
\label{eqBellmanforop2low}
(\BC_{low} f)(z)=\max_{\ga}\min_{\{\xi^j \in [d_j,u_j]\}}
[f(\xi \circ z)-(\ga, \xi \circ z-\rho z)].
\end{equation}

In this simple interval model and for convex $f$ this expression is trivial, it equals
$f(\rho z)$. On the other hand, if our $f$ is concave, or, more generally, if we allow only
finitely many jumps, which leads, instead of \eqref{eqBellmanforop2low}, to the operator

\begin{equation}
\label{eqBellmanforop3low}
(\BC_{low} f)(z)=\max_{\ga}\min_{\{\xi^j \in \{d_j,u_j\}\}}
[f(\xi \circ z)-(\ga, \xi \circ z-\rho z)],
\end{equation}
then Theorem \ref{thriskneutralmir} applies giving for the lower price
the dual expression to the upper price \eqref{eq1thbasicEuroprainbowBel}, where maximum is turned
 to minimum (over the same set of extreme risk-neutral measures):
\begin{equation}
\label{eq1thbasicEuroprainbowBellow}
(\BC_{low} f)(z)=\min_{\{\Om\}} \E_{\Om} f(\xi \circ z), \quad z=(z^1,\cdots, z^J).
\end{equation}

The difference between lower and upper prices can be considered as a measure
of intrinsic risk of an incomplete market.

\section{Continuous time limit}
\label{secconttime}

Our models and results are most naturally adapted to discrete time setting, which is
not a disadvantage from the practical point of view, as all concrete calculations
are anyway carried out on discrete data. However, for qualitative analysis, it is desirable to
be able to see what is going on in continuous time limit. This limit can also be simpler
sometimes and hence be used as an approximation to a less tractable discrete model.
Having this in mind, let us analyze possible limits as the time between jumps and their sizes tend
to zero.

Let us work with the general model of nonlinear jumps from Section \ref{secpathdep},
with the reduced Bellman operator of form \eqref{eqBellmanforopnonljump}. Suppose the maturity time is $T$.
Let us decompose the planning time $[0,T]$ into $n$ small intervals of length $\tau=T/n$, and assume
\begin{equation}
\label{eq1contlimop}
g_i(z)=z+\tau^{\al} \phi_i(z), \quad i=1, \cdots, k,
\end{equation}
with some functions $\phi_i$ and a constant $\al \in [1/2, 1]$. Thus the jumps during time $\tau$ are of the order
of magnitude $\tau^{\al}$. As usual, we assume that the risk free interest rate per time $\tau$ equals
\[
\rho =1+r\tau,
\]
with $r>0$.

From \eqref{eqBellmanforopnonljump} we deduce for the one-period Bellman operator the expression
\begin{equation}
\label{eq2contlimop}
\BC_{\tau} f (z)= \frac{1}{1+r\tau} \max_I \sum_{i\in I} p_i^I(z,\tau) f (z+\tau^{\al} \phi _i (z)),
\end{equation}
where $I$ are subsets of $\{1, \cdots, n\}$ of size $|I|=J+1$ such that the family of vectors
$z+\tau^{\al} \phi _i (z)$, $i\in I$, are in general position and $\{p_i^I(z,\tau)\}$ is the risk neutral
probability law on such family, with respect to $\rho z$, i.e.
\begin{equation}
\label{eq3contlimop}
\sum_{i\in I} p_i^I(z,\tau) (z+\tau^{\al} \phi _i (z))=(1+r \tau)z.
\end{equation}

 Let us deduce the HJB equation for the limit, as $\tau \to 0$, of the approximate cost-function
$\BC_{\tau}^{T-t}$, $t\in [0,T]$, with a given final cost $f_T$,
using the standard (heuristic) dynamic programming approach. Namely,
from \eqref{eq2contlimop} and assuming appropriate smoothness of $f$ we obtain the approximate equation
\[
 f_{t-\tau} (z)= \frac{1}{1+r\tau} \max_I \sum_{i\in I} p_i^I(z,\tau)
 \]
 \[
 \left[f_t(z)+\tau^{\al} \frac{\pa f_t}{\pa z} \phi_i(z)
 +\frac12 \tau^{2\al} \left (\frac{\pa^2 f_t}{\pa z^2} \phi _i (z), \phi_i(z)\right)
 + O(\tau^{3\al})\right].
\]

Since $\{p_i^I\}$ are probabilities and using \eqref{eq3contlimop}, this rewrites as
\[
 f_t-\tau \frac{\pa f_t}{\pa t} +O(\tau^2)= \frac{1}{1+r\tau}[f_t(z)+
 r \tau (z,\frac{\pa f_t}{\pa z})
 \]
 \[
 + \frac12 \tau^{2\al}\max_I \sum_{i\in I} p_i^I(z)
 \left (\frac{\pa^2 f_t}{\pa z^2} \phi _i (z), \phi_i(z)\right)]
 + O(\tau^{3\al}),
\]
where
\[
p_i^I(z)=\lim_{\tau \to 0} p_i^I (z,\tau)
\]
(clearly well defined non-negative numbers).
This leads to the following equations:
\begin{equation}
\label{eq4contlimop}
rf= \frac{\pa f}{\pa t}+r(z,\frac{\pa f}{\pa z})
+\frac12 \max_I \sum_{i\in I} p_i^I(z)
 \left (\frac{\pa^2 f}{\pa z^2} \phi _i (z), \phi_i(z)\right)
\end{equation}
in case $\al=1/2$, and to the trivial first order equation
\begin{equation}
\label{eq5contlimop}
rf= \frac{\pa f}{\pa t}+r(z,\frac{\pa f}{\pa z})
\end{equation}
with the obvious solution
\begin{equation}
\label{eq6contlimop}
f(t,z)= e^{-r(T-t)} f_T (e^{-r(T-t)}z)
\end{equation}
in case $\al >1/2$.

Equation \eqref{eq4contlimop} is a nonlinear extension of the classical Black-Scholes equation.
Well posedness of the Cauchy problem for such a nonlinear parabolic equation in the class of viscosity
solutions is well known in the theory of  controlled diffusions, as well as the fact that the solutions
solve the corresponding optimal control problem, see e.g. Fleming and Soner \cite{FleSo}.

\begin{remark} Having this well posedness, it should not be difficult to prove
the convergence of the above approximations rigorously, but I did not find the precise reference.
Moreover, one can be also interested in path-wise approximations.
For this purpose a multidimensional extension of the approach from
Bick and Willinger \cite{BickWi} (establishing path-wise convergence of Cox-Ross-Rubinstein
binomial approximations to the trajectories underlying the standard Black-Scholes equation
in a non-probabilistic way) would be quite relevant.
\end{remark}

In case $J=1$ and the classical CCR (binomial) setting with
\[
\sqrt {\tau}\phi_1=(u-1)z= \si \sqrt {\tau} z, \quad
 \sqrt {\tau}\phi_2=(d-1)z= -\si \sqrt {\tau} z,
 \]
 equation \eqref{eq4contlimop} turns to the usual Black-Scholes.

More generally, if $k=J+1$, the corresponding market in discrete time becomes complete (as noted
at the end of Section \ref{secpathdep}). In this case equation \eqref{eq4contlimop} reduces to
\begin{equation}
\label{eq61contlimop}
rf= \frac{\pa f}{\pa t}+r(z,\frac{\pa f}{\pa z})
+\frac12 \sum_{i=1}^{J+1} p_i(z)
 \left (\frac{\pa^2 f}{\pa z^2} \phi _i (z), \phi_i(z)\right),
\end{equation}
which is a generalized Black-Scholes equation describing a complete market (with randomness coming from $J$
correlated Brownian motions), whenever the diffusion matrix
\[
(\si^2)_{jk}=\sum_{i=1}^{J+1} p_i(z) \phi _i^j (z), \phi_i^k (z).
\]

As a more nontrivial example, let us consider the case of $J=2$ and a sub-modular final payoff $f_T$,
so that Theorem \ref{twostocks} applies to the approximations $\BC_{\tau}$.
Assume the simplest (and usual) symmetric form for upper and lower jumps
(further terms in Taylor expansion are irrelevant for the limiting equation):
 \begin{equation}
\label{eq7contlimop}
 u_i=1+\si_i \sqrt {\tau}, \quad  d_i=1-\si_i \sqrt {\tau}, \quad i=1,2.
 \end{equation}
 Hence
 \[
 \frac{u_i-\rho}{u_i-d_i}=\frac12 -\frac{r}{2\si_i}\sqrt {\tau}, \quad i=1,2,
 \]
 and
 \[
 \kappa =-\frac12 r\sqrt {\tau} (\frac{1}{\si_1}+\frac{1}{\si_2}).
 \]
As $\kappa <0$, we find ourselves in the second case of Theorem \ref{twostocks}.
Hence the only eligible collection of three vectors $\phi$ is $(d_1,u_2),(u_1,d_2),(u_1,u_2)$,
and the probability law $p_i^I$ is $(1/2, 1/2,0)$. Therefore, equation \eqref{eq4contlimop}
takes the form
 \begin{equation}
\label{eq8contlimop}
rf= \frac{\pa f}{\pa t}+r(z,\frac{\pa f}{\pa z})
+\frac12 \left[  \si_1^2 z_1^2 \frac{\pa^2 f}{\pa z_1^2}
-2  \si_1  \si_2 z_1 z_2 \frac{\pa^2 f}{\pa z_1 \pa z_2} + \si_2^2 z_2^2 \frac{\pa^2 f}{\pa z_2^2} \right].
\end{equation}

 The limiting Black-Scholes type equation is again linear in this example, but with degenerate
 second order part. In the analogous stochastic setting, this degeneracy would mean that only one
 Brownian motion is governing the behavior of both underlying stocks.
 This is not surprising in our approach, where Nature was assumed to be a single player.
 One could expect uncoupled second derivatives (non-degenerate diffusion) in the limit, if one would
 choose two independent players for the Nature, each playing for each stock.

Thus we are still in the setting of an incomplete market. The hedge price calculated from
equation \eqref{eq8contlimop} is actually the upper price, in the terminology of
Section \ref{secupperlowerprice}. To get a lower price, we shall use approximations
of type \eqref{eq1thbasicEuroprainbowBellow}, leading, instead of
\eqref{eq4contlimop}, to the equation
 \begin{equation}
\label{eq9contlimop}
rf= \frac{\pa f}{\pa t}+r(z,\frac{\pa f}{\pa z})
+\frac12 \min_I \sum_{i\in I} p_i^I(z)
 \left (\frac{\pa^2 f}{\pa z^2} \phi _i (z), \phi_i(z)\right).
\end{equation}

If $J=2$ the payoff is submodule, the maximum can be taken over the triples
  $(d_1,d_2)$, $(d_1,u_2)$, $(u_1,d_2)$ or $(d_1,u_2),(u_1,d_2),(u_1,u_2)$
(under \eqref{eq7contlimop} only the second triple works). Similarly
the minimum can be taken only over the triples
  $(d_1,d_2),(d_1,u_2),(u_1,u_2)$ or $(d_1,d_2)$, $(u_1,d_2)$, $(u_1,u_2)$.
  Under \eqref{eq7contlimop} both these cases give the same
  limit as $\tau\to 0$, yielding for the lower price the equation
   \begin{equation}
\label{eq10contlimop}
rf= \frac{\pa f}{\pa t}+r(z,\frac{\pa f}{\pa z})
+\frac12 \left[  \si_1^2  z_1^2 \frac{\pa^2 f}{\pa z_1^2}
+2  \si_1  \si_2 z_1 z_2 \frac{\pa^2 f}{\pa z_1 \pa z_2}
 + \si_2^2 z_2^2  \frac{\pa^2 f}{\pa z_2^2} \right],
\end{equation}
that differs only by sign at the mixed derivative from the equation
for the upper price.

As $f$ was assumed sub-modular, so that its mixed second derivative is negative, we have
\[
 \si_1  \si_2 \frac{\pa^2 f}{\pa z_1 \pa z_2} \le 0 \le - \si_1  \si_2 \frac{\pa^2 f}{\pa z_1 \pa z_2}.
 \]
 Hence, for the solution $f_u$ of the upper value equation \eqref{eq8contlimop}, the solution
 $f_l$ of the lower value equation \eqref{eq10contlimop}, and the solution $f_c$ of the classical
 Black-Scholes equation of a complete market based on two independent Brownian motions, i.e.
 equation \eqref{eq8contlimop} or \eqref{eq10contlimop} without the term with the mixed derivative
 (with the same sub-modular initial condition $f_T$) we have the inequality
 \[
 f_l \le f_c \le f_u,
 \]
  as expected.

Equations \eqref{eq8contlimop} and \eqref{eq10contlimop} can be solved explicitly via Fourier transform,
 just as the standard
Black-Scholes equation. Namely, changing the unknown function $f$ to $g$ by
\[
f(z_1,z_2)=e^{-r(T-t)}g(\frac{1}{\si_1}\log z_1, \frac{1}{\si_2}\log z_2),
\]
so that
\[
\frac{\pa f}{\pa z_i}= e^{-r(T-t)}\frac{1}{\si_i z_i}
\frac{\pa g}{\pa y_i}(\frac{1}{\si_1}\log z_1, \frac{1}{\si_2}\log z_2),
\]
transforms
these equations to the equations
   \begin{equation}
\label{eq11contlimop}
\frac{\pa g}{\pa t}+\frac12 (2r-\si_1)\frac{\pa g}{\pa y_1} +\frac12 (2r-\si_2)\frac{\pa g}{\pa y_2}
+\frac12 \left[ \frac{\pa^2 g}{\pa y_1^2}
\mp 2  \frac{\pa^2 g}{\pa y_1 \pa y_2}
 + \frac{\pa^2 g}{\pa y_2^2} \right]=0
\end{equation}
(with $\mp$ respectively).
Equation \eqref{eq11contlimop} has constant coefficients and
the equation for the Fourier transform $\tilde g (p)$ of $g$ is obviously
  \begin{equation}
\label{eq12contlimop}
\frac{\pa \tilde g}{\pa t}
=\frac12 [ (p_1 \mp p_2)^2-i(2r-\si_1) p_1-i(2r-\si_2) p_2] \tilde g.
\end{equation}
Hence the inverse Cauchy problem for equation \eqref{eq11contlimop} with a given
final function $g_T$ equals the convolution of $g_T$ with the inverse Fourier
transform of the functions
\[
\exp \{ -\frac12 (T-t) [(p_1 \mp p_2)^2-i(2r-\si_1) p_1-i(2r-\si_2) p_2] \},
\]
which equal (after changing the integration variables $p_1$ and $p_2$ to $q_1=p_1-p_2$,
$q_2=p_1+p_2$)
\[
\frac{1}{2(2\pi)^2}\int_{\R^2} dq_1 dq_2 \exp \{
-\frac12 (T-t)q_{1,2}^2
\]
\[
 +\frac{iq_1}{2}\left(y_1-y_2- \frac{(\si_1-\si_2)(T-t)}{2}\right)
+\frac{iq_2}{2}\left(y_1+y_2+(2r - \frac{\si_1+\si_2}{2})(T-t)\right)
\}
\]
(with $q_{1,2}$ corresponding to $\mp$), or explicitly
\[
\frac12 \frac{1}{\sqrt {2\pi (T-t)}} \de \left(\frac{y_1+y_2}{2}+(r-\frac{\si_1+\si_2}{4})(T-t)\right)
\exp \{ -\frac{1}{8(T-t)}\left(y_1-y_2-\frac{(\si_1-\si_2)(T-t)}{2}\right)^2
 \}
\]
and
\[
\frac12 \frac{1}{\sqrt {2\pi (T-t)}} \de \left(\frac{y_1-y_2}{2}-\frac{(\si_1-\si_2)(T-t)}{4}\right)
\exp \{ -\frac{1}{8(T-t)}\left(y_1+y_2+(2r-\frac{\si_1+\si_2}{2})(T-t)\right)^2
 \}
\]
respectively, where $\de$ denotes the Dirac $\de$-function.
Returning to equations \eqref{eq8contlimop} and \eqref{eq10contlimop} we conclude that
the the solutions $f_u$ and $f_l$ respectively of the inverse time Cauchy problem for these equations
are given by the formula
 \begin{equation}
\label{eq13contlimop}
f_{u,l}(t,z_1,z_2)=\int_0^{\infty} \int_0^{\infty} G^{\mp}_{T-t} (z_1,z_2;w_1,w_2) f_T(w_1,w_2) \, dw_1 dw_2,
\end{equation}
with the Green functions or transition probabilities being
\[
G^-_{T-t} (z_1,z_2;w_1,w_2)
\]
\[
=\frac{e^{-r(T-t)}}{2\sqrt {2\pi (T-t)}\si_1 \si_2 w_1 w_2}
\de \left(\frac{1}{2\si_1} \log \frac{z_1}{w_1}+\frac{1}{2\si_2} \log \frac{z_2}{w_2}
 +(r-\frac{\si_1+\si_2}{4})(T-t)\right)
\]
\begin{equation}
\label{eq14contlimop}
\exp \{ -\frac{1}{8(T-t)}\left(\frac{1}{\si_1} \log \frac{z_1}{w_1}-\frac{1}{\si_2} \log \frac{z_2}{w_2}
-\frac{(\si_1-\si_2)(T-t)}{2}\right)^2
\}
\end{equation}
and
\[
G^+_{T-t} (z_1,z_2;w_1,w_2)
\]
\[
=\frac{e^{-r(T-t)}}{2\sqrt {2\pi (T-t)}\si_1 \si_2 w_1 w_2}
\de \left(\frac{1}{2\si_1} \log \frac{z_1}{w_1}-\frac{1}{2\si_2} \log \frac{z_2}{w_2}
-\frac{(\si_1-\si_2)(T-t)}{4}\right)
\]
 \begin{equation}
\label{eq15contlimop}
\exp \{ -\frac{1}{8(T-t)}\left(\frac{1}{\si_1} \log \frac{z_1}{w_1}+\frac{1}{\si_2} \log \frac{z_2}{w_2}
 +(2r-\frac{\si_1+\si_2}{2})(T-t)\right)^2
\}
\end{equation}
respectively. Of course, formulas \eqref{eq13contlimop} can be further simplified by integrating over the
$\de$-function. Singularity, presented by this $\de$-function, is due to the degeneracy of the second order
part of the corresponding equations.

\section{Transaction costs in continuous time}
\label{secconttimetr}

The difficulties with transaction costs are well known in the usual stochastic analysis
approach, see e.g. Soner et al \cite{SSC} and
Bernard et al \cite{Bern07}.

In our approach, Theorem \ref{thtranscost} poses strong restrictions
 for incorporating transaction costs in
a continuous limit. In particular, assuming jumps of size $\tau^{\al}$ in a period of length
$\tau$, i.e. assuming \eqref{eq1contlimop}, only $\al=1$ can be used for the limit
$\tau \to 0$, because $\de_n(z)$ is of order $(1+\tau^\al)^n$, which tends to $\infty$, as $\tau=T/n \to 0$,
whenever $\al <1$. We know that for vanishing costs, assuming $\al =1$ leads to the trivial limiting equation
\eqref{eq5contlimop}, which was observed by many authors, see e.g. Bernard \cite{Bern05},
 McEneaney \cite{Mc}, Olsder \cite{Ol}.
 However, with transaction costs included, the model with jumps of order $\tau$ becomes not so obvious,
 but leads to a meaningful and manageable continuous time limit. To see this, assume that we are in the setting
 of Section \ref{secconttime} and transaction cost are specified, as in Section \ref{sectransco}, by
 a function $g$ satisfying \eqref{eqLipcosts}. To write a manageable approximation, we shall apply the following trick:
 we shall count at time $\tau m$ the transaction costs incurred at time $\tau(m+1)$
(the latter shift in transaction costs collection does not change, of course, the limiting process).
  Instead of  \eqref{eq2contlimop} we then get
 \[
\BC_{\tau} f (z)= \frac{1}{1+r\tau}
\max_I \sum_{i\in I} p_i^I(z,\tau)
\]
 \begin{equation}
\label{eq2contlimoptr}
\left[f (z+\tau^{\al} \phi _i (z))
+g(\ga (z+\tau \phi_i (z),\tau)-\ga (z,\tau), z+\tau \phi_i (z))\right],
\end{equation}
where $\ga (z,\tau)$ is the optimal $\ga$ chosen in the position $z$.
Assuming $g$ is differentiable, expanding and keeping the main terms, yields the following extension
of equation \eqref{eq5contlimop}:
\begin{equation}
\label{eq5contlimoptr}
rf= \frac{\pa f}{\pa t}+r(z,\frac{\pa f}{\pa z})+\psi (z),
\end{equation}
where
\[
\psi (z)=\max_I \sum_{i\in I} p_i^I(z)
 \sum_{m,j=1}^J \frac{\pa g}{\pa \ga^m}(\ga (z)) \frac{\pa \ga^m}{\pa z^j} \phi_i^j (z),
 \]
 with $\ga (z)=\lim_{\tau \to 0} \ga (z,\tau)$.

 This is a non-homogeneous equation with the corresponding homogeneous equation being
\eqref{eq5contlimop}. Since the (time inverse) Cauchy problem for
this homogeneous equation has the explicit solution \eqref{eq6contlimop},
we can write the explicit solution for the Cauchy problem of equation \eqref{eq5contlimoptr}
using the standard Duhamel principle (see e.g. \cite{KoMarbook}) yielding
\begin{equation}
\label{eq6contlimoptr}
f(t,z)= e^{-r(T-t)} f_T (e^{-r(T-t)}z)
 +\int_t^T e^{-r(s-t)} \psi (e^{-r(s-t)}z) \, ds.
\end{equation}
The convergence of the approximations $\BC_{\tau}^{[t/\tau]} f_T$ to this solution
 of equation \eqref{eq5contlimoptr} follows from the known general properties
 of the solutions to the HJB equations, see e.g. \cite{KolMas}.

Of course, one can also write down the modified equation
\eqref{eq4contlimop} obtained by introducing the transaction costs in the same way as above.
It is the equation
\[
rf= \frac{\pa f}{\pa t}+r(z,\frac{\pa f}{\pa z})
+\frac12 \max_I \sum_{i\in I} p_i^I(z)
\]
\begin{equation}
\label{eq4contlimoptr}
 \left[\left (\frac{\pa^2 f}{\pa z^2} \phi _i (z), \phi_i(z)\right)
 +\sum_{m,j=1}^J \frac{\pa g}{\pa \ga^m}(\ga (z)) \frac{\pa \ga^m}{\pa z^j} \phi_i^j (z)
 \right].
\end{equation}

However, as already mentioned, due to the restrictions of Theorem \ref{thtranscost},
only the solutions to a finite difference approximations of equation \eqref{eq4contlimoptr}
(with bounded below time steps $\tau$) represent justified hedging prices. Therefore our
model suggests natural bounds for time-periods between re-locations of capital, when transaction
costs remain amenable and do not override, so-to-say, hedging strategies. Passing to the limit
 $\tau \to 0$ in this model (i.e. considering continuous trading), does not lead to equation \eqref{eq4contlimoptr}, but to the
 trivial strategy of keeping all the capital on the risk free bonds. This compelled triviality is
 of course well known in the usual stochastic setting, see e. g.
Soner, Shreve and Cvitani\'c \cite{SSC}.

\section{Fractional dynamics}
\label{secfracdyn}

Till now we have analyzed the models where the jumps (from a given set) occur with regular
frequency. However, it is natural to allow the periods between jumps to be more flexible.
One can also have in mind an alternative picture of the model: instead of instantaneous jumps
at fixed periods one can think about waiting times for the distance from a previous price to reach
certain levels. It is clear then that these periods do not have to be constant. In the absence of a
detailed model, it is natural to take these waiting times as random variables. In the
simplest model, they can be i.i.d.  Their intensity represents
a kind of stochastic volatility.  Slowing down the waiting periods is, in some sense, equivalent
to decreasing the average jump size per period.

Assume now for simplicity that we are dealing with 2 colored options and sub-modular payoffs,
so that Theorem \ref{twostocks} applies yielding a unique eligible risk-neutral measure.
Hence the changes in prices (for investor choosing the optimal $\ga$)
follow the Markov chain $X_n^{\tau}(z)$ described by the recursive equation
\[
X_{n+1}^{\tau}(z)=X_n^{\tau}(z)+ \sqrt {\tau} \phi (X_n^{\tau}(z)), \quad X_0^{\tau}(z)=z,
\]
where $\phi (z)$ is one of three points $(z^1d_1,z^2u_2),(z^1u_1,z^2d_2),(z^1u_1,z^2u_2)$ that
are chosen with the corresponding risk neutral probabilities.
As was shown above, this Markov chain converges, as $\tau \to 0$ and $n=[t/\tau]$
(where $[s]$ denotes the integer part of a real number $s$), to the diffusion process $X_t$
solving the Black-Scholes type (degenerate) equation \eqref{eq8contlimop}, i.e. a sub-Markov process
with the generator
\begin{equation}
\label{eq8contlimopgener}
Lf(x)=-rf+r(z,\frac{\pa f}{\pa z})
+\frac12 \left[  \si_1^2 z_1^2  \frac{\pa^2 f}{\pa z_1^2}
-2  \si_1  \si_2 z_1 z_2 \frac{\pa^2 f}{\pa z_1 \pa z_2}
+ \si_2^2 z_2^2  \frac{\pa^2 f}{\pa z_2^2} \right].
\end{equation}

Assume now that the times between jumps $T_1,T_2,\cdots$ are i.i.d. random variables
with a power law decay, that is
\[
\P (T_i \ge t) \sim \frac{1}{\be t^{\be}}
\]
with $\be \in (0,1)$ (where $\P$  denotes probability and $\sim$ means, as usual, that the ratio between
the l.h.s. and the r.h.s. tends to one as $t \to \infty$).
It is well known that such $T_i$ belong to the domain of attraction of the $\be$-stable law (see e.g.
Uchaikin and Zolotarev \cite{UZ}) meaning that the normalized sums
\[
\Theta _t^{\tau}=\tau ^{\be} (T_1+\cdots +T_{[t/\tau]})
\]
(where $[s]$ denotes the integer part of a real number $s$)
converge, as $\tau \to 0$, to a $\be$-stable L\'evy motion $\Theta _t$, which is
a L\'evy process on $\R_+$ with the fractional derivative of order $\be$ as the generator:
\[
Af(t)=-\frac{d^{\be}}{d(-t)^{\be}}f(t)
=-\frac{1}{\Gamma (-\be)} \int_0^{\infty} (f(t+r)-f(t)) \frac{1}{y^{1+\be}} dr.
\]

 We are now interested in the process $Y^{\tau}(z)$ obtained from $X_n^{\tau}(z)$
 by changing the constant times between jumps by scaled random times $T_i$, so that
 \[
Y_t^{\tau}(z)=X^{\tau}_{N_t^{\tau}}(z),
\]
where
\[
N_t^{\tau}=\max \{ u: \Theta _u^{\tau} \le t\}.
\]
The limiting process
\[
N_t=\max \{ u: \Theta _u \le t\}
\]
is therefore the inverse (or hitting time) process of the $\be$-stable L\'evy motion $\Theta_t$.

By Theorem 4.2 and 5.1 of Kolokoltsov \cite{Ko09}, (see also Chapter 8 in \cite{KoMarbook}),
we obtain the following result.

\begin{theorem}
\label{thBlackScholesfrac}
The process $Y_t^{\tau}$ converges (in the sense of distribution on paths)
to the process $Y_t=X_{N_t}$, whose averages
 $f(T-t,x)=\E f (Y_{T-t}(x))$, for continuous bounded $f$,
 have the explicit integral representation
 \[
 f(T-t,x)= \int_0^{\infty} \int_0^{\infty} \int_0^{\infty}
 G^-_{u} (z_1,z_2; w_1,w_2)Q(T-t, u) \, du dw_1 dw_2,
 \]
 where $G^-$, the transition probabilities of $X_t$, are defined by
 \eqref{eq14contlimop}, and where $Q(t,u)$ denotes the probability density of the
 process $N_t$.

 Moreover, for $f\in C_{\infty}^2(\R^d)$, $f(t,x)$ satisfy the (generalized)
 fractional evolution equation (of Black-Scholes type)
\[
 \frac{d^{\be}}{dt^{\be}}f(t,x)= L f(t,x)+\frac{t^{-\be}}{\Gamma (1-\be)} f(t,x).
 \]
\end{theorem}

\begin{remark}
Similar result to Theorem 4.2 of Kolokoltsov \cite{Ko09} used above,
but with position independent random walks, i.e. when $L$ is the generator of a L\'evy process,
were obtained in Meerschaert and Scheffler \cite{MS}, see also related results
in Kolokoltsov, Korolev and Uchaikin \cite{KoKoUch}, Henry,  Langlands and Straka
\cite{HeLaSt} and in references therein. Rather general fractional Cauchy problems
are discussed in Kochubei \cite{Koch08}.
\end{remark}

Similar procedure with a general nonlinear Black-Scholes type equation
\eqref{eq4contlimop} will lead of course to its similar fractional extension.
However, a rigorous analysis of the corresponding limiting procedure is beyond the scope
of the present paper.

\end{document}